\documentclass[11pt]{article}
\usepackage{amsmath,amssymb,amsthm}
\usepackage{geometry}
\usepackage{mathtools}
\usepackage{mathrsfs}
\usepackage[hidelinks]{hyperref}
\usepackage{enumitem} 
\numberwithin{equation}{section}  
\usepackage{cite} 

\newtheorem{theorem}{Theorem}[section]
\newtheorem{lemma}{Lemma}[section]

\newtheorem{remark}{Remark}[section]
\newtheorem{definition}{Definition}[section]

\title{Renormalized Solution for the Nonlinear Parabolic Problem with Lower Order Terms}
\author{
	LI Shijun\thanks{Email: sjlee@hainanu.edu.cn} \\
	School of Mathematics and Statistics, Hainan University, Haikou, China \\
	\and
	LI Shujing \\
	School of Mathematics and Statistics, Hainan University, Haikou, China \\
	\and
	XU Shaopeng\thanks{Corresponding author: xuxsp@126.com} \\
	School of Mathematics and Statistics, Hainan University, Haikou, China 
}
\date{\notag}
\begin{document}
	\maketitle
	\noindent \textbf{Abstract:} 
	In this paper, we consider the following nonlinear parabolic equation with non-coercive terms in \(R^N\) space
	\[
	\dfrac{\partial u}{\partial t} -\nabla \cdot (a(x,t,u,\nabla u)+ \Phi(x,t,\nabla u))=f, \text{ in }\Omega \times (0,T).
	\]
	Here \(\Omega\) is a bounded open set of \(R^N\) with the boundary \(\partial \Omega\) satisfying Lipschitz condition. The Carathéodory function \(\Phi\) is restricted by $|\Phi(x,t,s)|\le c(x,t)|s|^\gamma$ with parameters depending on $p$ and $N$. And the initial value $u(x,0)=u_0(x)$. For convenience, we define the domain $Q\coloneqq\Omega \times (0,T)$ and the boundary similarly. Then for $f\in L^1(Q)$ and $u_0\in L^1(\Omega)$, we prove the existence and uniqueness of a renormalized solution via truncation methods, monotone operator theory, and a prior gradient estimates.
	\bigskip

	\noindent \textbf{Keywords:}  
	Renormalized solutions; Lower order terms; approximation equation; integrable data.
	
	\vspace{1cm}

\section{Introduction}

	In recent years, with the development of fluid mechanics, elasticity, image processing, and other disciplines, research on partial differential equations (PDEs) has increased significantly. However, in practical applications, the strong regularity requirements on initial data, boundary values, coefficients, and nonhomogeneous terms often limit the applicability of classical solutions. Even some simple equations may have no solution in the weak sense. When the initial data and right-hand side belong to \( L^1 \), weak solutions may not exist (e.g., the Boltzmann equation). To address this issue, DiPerna and Lions first introduced the concept of renormalized solutions for the Boltzmann equation \cite{DiPerna1989b} and established the well-posedness of solutions in the renormalized sense. This approach has since found important applications in hydrodynamic limits, nonlinear elliptic equations, and parabolic equations.
	
	The Boltzmann equation is a nonlinear integro-differential equation describing the statistical behavior of thermodynamic systems out of equilibrium. Its right-hand side is the collision operator. It describes the irreversible process of a dilute gas transitioning from an arbitrary state to equilibrium, i.e., the evolution of the molecular velocity distribution function. For a single gas component with spherically symmetric molecular interactions and an external force \( F \), the velocity distribution function \( f(x, v, t) \) satisfies
	\[
	\frac{\partial f}{\partial t} + v \cdot \nabla_x f + \frac{F}{m} \cdot \nabla_v f = \int (f' f_1' - f f_1) \, g \, b \, \mathrm{d}b \, \mathrm{d}s \, \mathrm{d}v_1 = Q(f, f),
	\]
	where the right-hand side is the collision integral, representing changes in \( f \) due to molecular collisions.
	
	In the 1980s, based on the properties of the Boltzmann equation, DiPerna and Lions studied the Cauchy problem for the equation
	\[
	\frac{\partial f}{\partial t} + \xi \cdot \nabla_x f = Q(f, f)
	\]
	with large initial data in \( (0, \infty) \times \mathbb{R}^N \times \mathbb{R}^N \) \cite{DiPerna1989b}. They proved that sequences of classical solutions with a priori bounds converge weakly in \( L^1(Q) \) to a renormalized solution and, from the stability results, derived the existence of global renormalized solutions under given initial conditions.
	
	The concept of renormalized solutions was introduced in 1989 by DiPerna et al. for first-order equations \cite{DiPerna1989a,DiPerna1989b}. Subsequently, it was developed for elliptic equations with \( L^1 \) data \cite{Murat1993} and for elliptic equations with bounded measure data \cite{DalMaso1999}, and extended to parabolic equations with \( L^1 \) data \cite{Blanchard1997,Blanchard2001}. In 1995, B\'enilan et al. introduced the concept of entropy solutions \cite{Prignet1997}, with similar results in the parabolic setting. Over time, more general nonlinear elliptic and parabolic equations have been considered. For example, in 2001, Blanchard et al. \cite{Blanchard2001} proved the existence of renormalized solutions for a nonlinear parabolic equation with a non-coercive lower-order term \( \operatorname{div}(\Phi(u)) \). In 2003, Boccardo et al. \cite{Boccardo2003} studied bounded and unbounded solutions for equations with \( u_0 \in L^1(Q) \) and \( E \in (L^2(Q))^N \). In 2010, Di Nardo \cite{DiNardo2010} proved the existence of renormalized solutions for an equation with \( f \in L^1(Q) \) and \( u_0 \in L^1(\Omega) \). In 2011, Di Nardo et al. \cite{DiNardo2011} further considered an additional non-coercive lower-order term \( b|\nabla u|^\delta \) and proved the existence and uniqueness of renormalized solutions.
	
	In this paper, we study the following nonlinear parabolic equation with a non-coercive term in \( \mathbb{R}^N \):
	\begin{equation}\label{equations}
		\begin{cases}
			\dfrac{\partial u}{\partial t} - \nabla\cdot\bigl(a( x, t, \nabla u) + \Phi( x, t, u)\bigr) = f, & (x, t) \in (0,T) \times \Omega, \\[6pt]
			\bigl(a( x,t, \nabla u) + \Phi( x,t, u)\bigr) \cdot \mathbf{n} = 0, & (t, x) \in Q_n, \\[4pt]
			u = 0, & (t, x) \in Q_d, \\[4pt]
			u(x, 0) = u_0(x), & x \in \Omega.
		\end{cases}
	\end{equation}
	where \( f \in L^1(Q) \), \( u_0 \in L^1(\Omega) \), \( \Omega \subset \mathbb{R}^N \) is a bounded open set with Lipschitz boundary, \( Q = \Omega \times (0, T) \), \( Q_n = \Gamma_n \times (0, T) \), \( Q_d = \Gamma_d \times (0, T) \), with \( \Gamma_n \cup \Gamma_d = \partial\Omega \), \( \Gamma_n \cap \Gamma_d = \emptyset \), and \( \sigma(\Gamma_d) > 0 \). The function \( \Phi : Q \times \mathbb{R} \to \mathbb{R}^N \) is a Carath\'eodory function satisfying \( |\Phi(x, t, s)| \leq c(x, t) |s|^\gamma \), with conditions depending on \( p \) and \( N \). The motivation for studying renormalized solutions to (3) comes from homogenization problems where \( \Omega \) is a perforated domain, with Neumann conditions on the boundary of the holes and Dirichlet conditions on the outer boundary. Thus, we consider a mixed boundary problem. The main novelty is that the term \( -\operatorname{div}\Phi(x, t, u) \) is non-coercive. When \( \Phi \equiv 0 \), Boccardo et al. \cite{Boccardo1989} proved the existence of weak solutions for parabolic problems with Dirichlet boundary conditions, and similar results hold for \( L^1 \) data or bounded measures when \( p > 2 - \frac{1}{N+1} \) \cite{Boccardo1997}. In those works, weak solutions belong to \( L^m(0, T; W_0^{1,m}(\Omega)) \) with \( m < \frac{p(N+1)-N}{N+1} \). However, existence and uniqueness of weak solutions are not guaranteed in general \cite{Prignet1995,Serrin1964}. To remove restrictions on \( p \) and ensure stability, we adopt the framework of renormalized solutions.
	
	In this paper, we first prove the existence of renormalized solutions for a related approximating equation. We derive a priori estimates for the gradients of the approximate solutions using techniques from \cite{Bottaro1973,Porzio1999}. We then verify the conditions for renormalized solutions and, using functional analysis and real variable methods, obtain the convergence of the approximate solutions, thereby establishing the existence of renormalized solutions. Finally, we prove the existence and uniqueness of renormalized solutions to problem (3). The existence proof follows a similar approach as for the approximating equation, and uniqueness is obtained by showing that two renormalized solutions with the same initial data satisfy the required estimates.

\section{Marks and Approximate Problem}
	In this paper, $q \in [1, \infty]$. The space $W_{\Gamma_d}^{1,q}(\Omega)$ denotes the space of functions in $W^{1,q}(\Omega)$ that vanish on $\Gamma_d \subset \partial \Omega$.
	
	Let $\Omega \subset \mathbb{R}^N$ ($N \geq 2$) be a bounded connected open subset with Lipschitz boundary and $\sigma(\Gamma_d) > 0$. The norm of $W_{\Gamma_d}^{1,q}(\Omega)$ is defined as
	\begin{equation}
		\|v\|_{W_{\Gamma_d}^{1,q}(\Omega)} = \|\nabla v\|_{L^q(\Omega)} \qquad \text{(see \cite{Ziemer1989})}.
	\end{equation}
\begin{remark}
	$a:Q\times R^N\rightarrow R^N $ is a Carathéodory function satisfying the following conditions:
	\begin{equation}\label{as1}
		a(x,t,\xi)\xi\geq \alpha|\xi|^p ,\quad \alpha >0,\quad a.e. (x,T)\in Q,
	\end{equation}
	\begin{equation}\label{as2}
		a(x,t,0)=0,
	\end{equation}
	\begin{equation}\label{as3}
		|a(x,t,\xi)|\leq k(b(x,t)+|\xi|^{p-1}), \quad a.e. (x,t)\in Q,\quad b(x,t)\in L^{p'}(Q), \quad k>0,
	\end{equation}
	\begin{equation}\label{as4}
		(a(x,t,\xi)-a(x,t,\xi'))\cdot (\xi-\xi')\geq 0,\quad \forall \xi, \xi' \in R^N, a.e. (x,t)\in Q,\quad \xi\neq\xi'.
	\end{equation}
\end{remark}
	and
\begin{remark}
	$\Phi: Q\times R^N\rightarrow R^N$ is a Carathéodory function satisfying the following conditions:
	\begin{equation}\label{as5}
		|\Phi(x,t,s)|\leq c(x,t)|s|^\gamma,\quad \forall s\in R,\quad a.e.(x,t)\in Q,
	\end{equation}
	\begin{equation}\label{as6}
		\gamma=\dfrac {N+2}{N+p}(p-1),\quad c(x,t)\in L^m(Q),\quad m=\dfrac{N+p}{p-1}.
	\end{equation}
		Therefore,
	\begin{equation}\label{as7}
		f\in L^1(Q) \text{ and } u_0\in L^1(\Omega)
	\end{equation}
	is held.
\end{remark}
\begin{definition}
	For any $k > 0$, define
	\[
	T_k(r) = \min\{k, \max\{r, -k\}\} =
	\begin{cases} 
		k, & r \geq k, \\
		r, & |r| < k, \\
		-k, & r \leq -k.
	\end{cases}
	\]
	\[
	\Theta_k(r) =
	\begin{cases} 
		\dfrac{r^2}{2}, & |r| < k, \\[10pt]
		k|r| - \dfrac{k^2}{2}, & |r| \geq k.
	\end{cases}
	\]
	Here $\Theta_k(r)$ is a primitive function of $T_k(r)$. Clearly, the functions defined above are Lipschitz continuous, and we have $T_k(u) \in L^p(0,T; W_{\Gamma_d}^{1,p}(\Omega))$,
	\[
	\nabla T_k(u) \in (L^p(Q))^N,
	\]
	with $|T_k| \leq k$, $|T_k(r)| \leq |r|$, $\Theta_k(r) \geq 0$, and $\Theta_k(r) \leq k |r|$.
\end{definition}
The notion of a very weak gradient is introduced below.

\begin{definition}
	Let $u$ be a measurable function defined on $Q$ such that for every $k > 0$, $T_k(u) \in L^p(0,T; W_{\Gamma_d}^{1,p}(\Omega))$. Then there exists a unique measurable function $v : Q \to \mathbb{R}^N$ satisfying
	\[
	\nabla T_k(u) = v \, \chi_{\{|u| < k\}} \quad \text{a.e. in } Q, \text{ for every } k > 0,
	\]
	where $\chi_E$ denotes the characteristic function of a measurable set $E$. The function $v$ is called a \textbf{very weak gradient} of $u$ and is denoted by $v = \nabla u$. If $u \in L^1(0,T; W_{\Gamma_d}^{1,1}(\Omega))$, then $v$ coincides with the weak derivative of $u$.
\end{definition}

Using the notion of a very weak gradient, one can define renormalized solutions.

\begin{definition}
	A measurable function $u$ defined on $\Omega \times (0,T)$ is called a \textbf{renormalized solution} to Problem~(3) if the following conditions hold:
	\begin{equation}
		u \in L^\infty(0,T; L^1(\Omega)), 
	\end{equation}
	\begin{equation}\label{beforetau0}
		T_k(u) \in L^p(0,T; W^{1,p}_{\text{1},p}(\Omega)), \quad \forall k > 0, 
	\end{equation}
	\begin{equation}\label{iintau0}
		\lim_{n \to \infty} \frac{1}{n} \iint_{\{(x,t) \in Q : |u| \leq n\}} a(x,t,\nabla u) \cdot \nabla u \mathrm{d}x\mathrm{d}t = 0.
	\end{equation}
	For any $\varphi \in C^1(\overline{Q})$ with $\varphi(x, T) = 0$, and for any pointwise $C^1$ function $S \in W^{2,\infty}(\mathbb{R})$ such that $S'$ has compact support, $u$ satisfies the following equality:
	\begin{equation}\label{longeqv}
		\begin{aligned}[b]
			-\int_{\Omega}\varphi(&x,0)S(u_0)\mathrm{d}x
			+\int^0_T \int_{\Omega}S(u)\frac{\partial \varphi}{\partial t}\mathrm{d}x\mathrm{d}t
			+\int^0_T \int_{\Omega}S'a(x,t,\nabla u)\nabla \varphi \mathrm{d}x\mathrm{d}t \\
			+\int^0_T& \int_{\Omega} S''(u)\varphi a(x,t,\nabla u)\nabla u \mathrm{d}x\mathrm{d}t 
			+\int^0_T \int_{\Omega}\int_{\Omega}S'(u)\Phi(x,t,u)\nabla \varphi \mathrm{d}x\mathrm{d}t \\ 
			+&\int^0_T \int_{\Omega}S''(u)\varphi\Phi(x,t,u)\nabla u \mathrm{d}x\mathrm{d}t \\
			&=\int^0_T \int_{\Omega}fS'(u)\varphi \mathrm{d}x\mathrm{d}t.
		\end{aligned}
	\end{equation}
\end{definition}
\begin{remark}
	In particular, there exists $M > 0$ such that $\operatorname{supp} S' \subset [-M, M]$ and
	\[
	\int_0^T \int_{\Omega} \Phi(x, t, u) S'(u) \nabla u\mathrm{d}x\mathrm{d}t= \int_0^T \int_{\Omega} \Phi(x, t, T_M(u)) S'(u) \nabla T_M(u) \mathrm{d}x\mathrm{d}t.
	\]
\end{remark}
To obtain a priori estimates for the approximate solutions to equation \eqref{equations}, we need a technical result. It provides estimates for $u$ and its gradient $\nabla u$ in the Lorentz spaces
\[
L^{\frac{p(N+1)-N}{N(p-1)}, \infty}
\text{ and }
L^{\frac{p(N+1)-N}{(N+1)(p-1)}, \infty},
\]
respectively. By the classical Lorentz space embeddings into Lebesgue spaces \cite{Hunt1966}, we then derive estimates for $u$ and $|\nabla u|$ in the Lebesgue spaces
\[
L^m(Q)
\text{ and }
L^s(Q)
\]
respectively, where
\[
m < \frac{p(N+1)-N}{N(p-1)}, \qquad s < \frac{p(N+1)-N}{(N+1)(p-1)}.
\] 
\begin{definition}[Equimeasurable Decreasing Rearrangement]
	Let $u$ be a measurable function that is finite almost everywhere. Its distribution function $\delta_u(t)$ is defined as
	\[
	\delta_u(t) = \mu\{ x \in \Omega : |u(x)| > t \}.
	\]
	Then $\delta_u(t)$ is monotonically decreasing on $[0, \infty)$.
\end{definition}
\begin{definition}[\cite{Adams1975}\cite{Lorentz1950}]
	If $u \in L^p(\Omega)$, then
	\begin{equation}
	\|u\|_{L^p(\Omega)} = 
	\begin{cases} 
		\left( p \displaystyle\int_0^\infty t^p \delta_u(t) \, \dfrac{\mathrm{d}t}{t} \right)^{\frac{1}{p}}, & 1 \leq p < \infty, \\[10pt]
		\,\inf\{ t : \delta_u(t) = 0 \}, & p = \infty.
	\end{cases}
	\end{equation}
	The \emph{equimeasurable decreasing rearrangement} $u^*$ of $u$ is defined as
	\[
	u^*(s) = \inf\{ t : \delta_u(t) \leq s \}.
	\]
	Since $\delta_u$ is monotonically decreasing, $u^*$ is also monotonically decreasing.
\end{definition}
\begin{definition}[\cite{Adams1975}\cite{Lorentz1950}]
	If $u$ is measurable in $\Omega$, then
	\[
	u^{**}(t) = \frac{1}{t} \int_0^t u^*(s) \, ds
	\]
	is the average value of $u^*$ over $[0,t]$. Since $u^*$ is monotonically decreasing, we have 
	$$u^*(t) \leq u^{**}(t).$$
	For $1 \leq p \leq \infty$, define the norm
	\begin{equation}
	\|u\|_{L^{p,q}(\Omega)} = 
	\begin{cases}
		\left( \displaystyle\int_0^\infty \left( t^{\frac{1}{p}} u^*(t) \right)^q \dfrac{\mathrm{d}t}{t} \right)^{\frac{1}{q}}, & 1 \leq q < \infty, \\[10pt]
		\,\sup\{ t^{\frac{1}{p}} u^{**}(t) : t>0\}, & q = \infty.
	\end{cases}
	\end{equation}
	The \textbf{Lorentz space} is the set of measurable functions $u$ on $\Omega$ such that $\|u\|_{L^{p,q}(\Omega)} < \infty$. Moreover, for $1 \leq p < \infty$ and $1 \leq q \leq \infty$, the Lorentz space is a Banach space.
\end{definition}
\begin{lemma}[\cite{DiNardo2010}]\label{lemmadinardo}
	Let $\Omega$ be an open set in Euclidean space $R^N$ with finite measure. $u$ is a measurable function satisfying for $\forall k>0$:
	\begin{equation}
		T_{k(u)\in L^{p}(0,T;W^{1,p}_{\Gamma_{d}}(\Omega))}\bigcap L^{\infty}(0,T;L^{2}(\Omega)),
	\end{equation}
	such that
	\begin{equation}
		\sup_{t\in(0,T)}\int_{\Omega}\left| T_{k}(u)\right|^{2}\mathrm{d}x+\int_{0}^{T}\int_{\Omega}\left| \nabla T_{k}(u)\right|^{p}\leq Mk,
	\end{equation}
	\begin{equation}
		\sup_{{t\in(0,T)}}\int_{\Omega}|T_k(u(t))|^2\mathrm{d}x+\int_0^T\int_{\Omega}|\nabla T_k(u)|^p \mathrm{d}x\mathrm{d}t\leq kM+L,
	\end{equation}
	where \(M\) and \(L\) are positive constant. Sharply we can see that
	\begin{equation}
		\left\||u|^{p-1}\right\| _{L^{\frac{p(N+1)-N}{N(p-1)},\infty}(Q)}\leq C(N,p)M^{\frac{(p-1(N+p))}{p(N+1)-N}}|Q|^{\frac{1}{p'}\frac{N}{N+p'}},
	\end{equation}
	and
	\begin{equation}
		\||\nabla u|^{p-1}\|_{L^{\frac{p(N+1)-N}{(N+1)(p-1)},\infty}(Q)}\leq C(N,p)M^{\frac{(N+2)(p-1)}{p(N+1)-N}}.
	\end{equation}
\end{lemma}
\begin{lemma}[Gagliardo--Nirenberg inequality \cite{Nirenberg1959}]
	Let $u \in L^q(\Omega)$ and assume that its $m$-th order derivative $\nabla^m u$ belongs to $L^\gamma(\Omega)$, with $q \ge 1$ and $\gamma \le \infty$. Then for every $j$ with $0 \le j < m$, the following inequality holds:
	\[
	\| \nabla^j u \|_{L^p(\Omega)} \;\le\; C \; \| \nabla^m u \|_{L^\gamma(\Omega)}^{a} \; \| u \|_{L^q(\Omega)}^{1-a},
	\]
	where
	\[
	\frac{1}{p} \;=\; \frac{j}{n} \;+\; a\Bigl(\frac{1}{\gamma} - \frac{m}{n}\Bigr) \;+\; (1-a)\frac{1}{q},
	\]
	and $C$ is a constant independent of $u$.
\end{lemma}
\begin{lemma}[\cite{Boccardo1988}\label{lemmaboccardo}]
	Assume that $u^\varepsilon$ converges weakly in $L^p(0,T; W^{1,p}_{\text{L}^d}(\Omega))$ or almost everywhere in $Q$ to $u$, and that
	\[
	\iint_Q \bigl[ a(x, t, \nabla u^\varepsilon) - a(x, t, \nabla u) \bigr] \nabla (u^\varepsilon - u) \to 0.
	\]
	Then
	\[
	\nabla u^\varepsilon \to \nabla u \quad \text{strongly in } (L^p(Q))^N.
	\]
\end{lemma}

	For $\forall \varepsilon>0$, we consider the following approximate problem
	\begin{equation}\label{approximateequation}
		\begin{cases}
			\dfrac{\partial u^\varepsilon}{\partial t}- \nabla\cdot(a(t,x,\nabla u^\varepsilon)+\Phi_\varepsilon (t,x,u^\varepsilon))=f^\varepsilon& \text{in }\Omega\times (0,T),\\
			(a(t,x,\nabla u^\varepsilon)+\Phi_\varepsilon(t,x,u^\varepsilon))\cdot \mathbf{n}=0 & \text{on } \Gamma _n \times (0,T) ,\\
			u^\varepsilon=0 & \text{on } \Gamma_d \times(0,T),\\
			u^\varepsilon(x,0)=u^\varepsilon_0(x)& \text{in } \Omega.
		\end{cases}
	\end{equation} 
	Here $\Phi _\varepsilon$ is defined as following
	\begin{equation}\label{phiapproximate}
		\Phi _\varepsilon(x,t,s)=\Phi (s,t,T_{\frac{1}{\varepsilon}}(s)),
	\end{equation}
	it satisfies
	\begin{equation}
		\begin{cases}
			|\Phi_{\varepsilon}(x,t,s)|\leq|\Phi(x,t,s)|\leq c(x,t)|s|^{\gamma},\\
			|\Phi_{\varepsilon}(x,t,s)|\leq c(x,t)(\frac{1}{\varepsilon})^{\gamma}.
		\end{cases}
	\end{equation}
	And
	\begin{equation}\label{fst}
		f^{\varepsilon}\in L^{p'}(Q),\quad f^{\varepsilon}\rightarrow f, \text{ strongly in }L^{1}(Q),
	\end{equation}
	\begin{equation}\label{ust}
		u^{\varepsilon}_{0}\in L^{2}(\Omega), \quad u^{\varepsilon}_{o}\rightarrow u_{0},\text{ strongly in } L^{1}(\Omega).
	\end{equation}
	Then by \cite{Lions1969}, it is easy to proof that there exists a weak solution $u^{\varepsilon}\in L^{p}(0,T;W^{1,p}_{\Gamma_{d}}(\Omega))$ of approximate problem \eqref{approximateequation}.

\section{Prior Estimate}
\begin{theorem}
	For $\forall \varepsilon>0$, we have
	\begin{equation}\label{tklpbnd}
		T_{k}(u^{\varepsilon}) \text{ is bounded in } L^{p}(0,T ;W^{1,p}_{\Gamma_{d}}(\Omega)).
	\end{equation}
\end{theorem}
\begin{proof}
		For $\forall k>0$, we test the first equation of \eqref{approximateequation} by $T_{k}(u^{\varepsilon})$ to get
	\begin{equation}\label{testingfunction}
		\begin{aligned}[b]
			&\int ^{T}_{0}\langle \frac {\partial u^\varepsilon}{\partial t},T_{k}(u^{\varepsilon})\rangle \mathrm{d}t+\int^{T}_{0}\int_{\Omega} a(x,t,\nabla u^{\varepsilon})\nabla T_{k}(u^{\varepsilon})\,\mathrm{d}x\mathrm{d}t +\int^{T}_{0}\int_{\Omega} \Phi _{\varepsilon}(x,t,u^{\varepsilon})\nabla T_{k}(u^{\varepsilon})\,\mathrm{d}x\mathrm{d}t \\
			=&\int^{T}_{0}\int_{\Omega} f^{\varepsilon}T_{k}(u^{\varepsilon})\,\mathrm{d}x\mathrm{d}t.
		\end{aligned}
	\end{equation}
	Since 
	\begin{equation}
		\int ^{T}_{0}\langle \dfrac {\partial u^\varepsilon}{\partial t},T_{k}(u^{\varepsilon})\rangle \,\mathrm{d}t
		=\int_{\Omega}\Theta_{k}(u^{\varepsilon})\,\mathrm{d}x
		-\int_{\Omega}\Theta_{k}(u^{\varepsilon}_{0})\,\mathrm{d}x,
	\end{equation}
	and the inequality
	\begin{equation}
		\frac{1}{2}|T_{k}(s)|^2\leq \frac{1}{2}sT_{k}(s)\leq \Theta_{k}(s)\leq k|s|,
	\end{equation}
	the following estimate holds
	\begin{equation}\label{eatimate}
		\int ^{T}_{0}\langle \frac {\partial u^\varepsilon}{\partial t},T_{k}(u^{\varepsilon})\rangle \,\mathrm{d}t\geq \frac{1}{2}\int_{\Omega}u^{\varepsilon}T_{k}(u^{\varepsilon})\,\mathrm{d}x-k\int_{\Omega}|u^{\varepsilon}_{0}|\,\mathrm{d}x.
	\end{equation}
	Substitute \eqref{as1}, \eqref{phiapproximate} and \eqref{eatimate} in to \eqref{testingfunction} then it results
	\begin{equation}\label{c1}
		\begin{aligned}[b]
			&\frac{1}{2}\int_{\Omega}u^{\varepsilon}T_{k}(u^{\varepsilon})\,\mathrm{d}x	+\alpha\int_{0}^{T}\int_{\Omega}\left|\nabla T_{k}(u^\varepsilon))\right|^{p}\,\mathrm{d}x\mathrm{d}t\\
			\leq& k\int_{0}^{T}\int_{\Omega}\left|f^\varepsilon\right|\,\mathrm{d}x\mathrm{d}t
			+k\int_{\Omega}\left|u^{\varepsilon}_{0}\right|\,\mathrm{d}x
			+\int_{0}^{T}\left|\Phi_{\varepsilon}(x,t,u^\varepsilon)\right|\left|\nabla T_{k}(u^{\varepsilon})\right|\,\mathrm{d}x\mathrm{d}t\\
			\leq& k C_{1}+\int^{T}_{0}\int_{\Omega}|\Phi(x,t,T_{k}(u^\varepsilon))||\nabla T_{k}(u^{\varepsilon})|\,\mathrm{d}x\mathrm{d}t\\
			\leq& kC_{1}+\int_{0}^{T}\int_{\Omega} c(x,t)\left|T_{k}(u^{\varepsilon})\right|^{\gamma}\left|\nabla T_{k}(u^\varepsilon)\right| \,\mathrm{d}x\mathrm{d}t.
		\end{aligned}
	\end{equation}
	Then, using Hölder's inequality, the Gagliardo-Nirenberg inequality and Young's inequality, we obtain
	\begin{equation}\label{c2c3}
		\begin{aligned}[b]
			&\int_{0}^{T}\int_{\Omega}c(x,t)\left|T_{k}(u^{\varepsilon})\right|^{\gamma}\left|\nabla T_{k}(u^{\varepsilon})\right|\,\mathrm{d}x\mathrm{d}t\\
			\leq& \left(\int_{0}^{T}\int_{\Omega}\left|c(x,t)\right|^{p'}\left|T_{k}(u^\varepsilon)\right|^{\gamma p'}\,\mathrm{d}x\mathrm{d}t\right)^{\frac{1}{p'}}\left(\int_{0}^{T}\int_{\Omega}\left|\nabla T_{k}(u^{\varepsilon})\right|^{p}\,\mathrm{d}x\mathrm{d}t\right)^{\frac{1}{p}}\\
			\leq&\left(\int_{0}^{T}\int_{\Omega}\left|c(x,t)\right|^mdxdt\right)^{\frac{1}{m}}\left(\int_{0}^{T}\int_{\Omega}
			\left|T_{k}(u^{\varepsilon})\right|^{\frac{p(N+2)}{N}}\,\mathrm{d}x\mathrm{d}t\right)^{\frac{N(p-1)}{p(N+p)}}
			\left(\int_{0}^{T}\int_{\Omega}\left|\nabla T_{k}(u^{\varepsilon})\right|^{p}\,\mathrm{d}x\mathrm{d}t\right)^{\frac{1}{p}}\\
			\leq& C_{2}\|c(x,t)\|_{L^{m}(Q)}\left(\sup_{t\in(0,T)}\int_{\Omega}\left|T_{k}(u^{\varepsilon})\right|^{2}\,\mathrm{d}x\right)^{\frac{1}{m}}
			\left(\int_{0}^{T}\int_{\Omega}\left|\nabla T_{k}(u^{\varepsilon})\right|^{p}\,\mathrm{d}x\mathrm{d}t\right)^{\frac{N+1}{N+p}}\\
			\leq& C_{3}k+\frac{N+1}{N+p}\int_{0}^{T}\int_{\Omega}\left|\nabla T_{k}(u^{\varepsilon})\right|^{p}\,\mathrm{d}x\mathrm{d}t,
		\end{aligned}
	\end{equation}
	where $C_{2}=C_{2}(N,p)$ and $ C_{3}=C_{3}(N,p,\|c(x,t)\|_{L^{m}(Q)})$. Through \eqref{c1} and \eqref{c2c3} we conclude
	\begin{equation*}
		\frac{1}{2}\sup_{t\in(0,T)}\int_{\Omega}\left|T_{k}(u^{\varepsilon})\right|^{2}\,\mathrm{d}x+\int_{0}^{T}
		\int_{\Omega}\left|\nabla T_{k}(u^{\varepsilon})\right|^{p}\,\mathrm{d}x\mathrm{d}t\leq Mk,
	\end{equation*}
	since
	\begin{equation*}
		\sup_{t\in(0,T)}\int_{\Omega}\left|T_{k}(u^{\varepsilon})\right|^2\,\mathrm{d}x\leq\sup_{t\in(0,T)}\int_{\Omega}u^{\varepsilon}T_{k}(u^{\varepsilon})\,\mathrm{d}x,
	\end{equation*}
	then
	\begin{equation}\label{mk}
		\frac{1}{2}\sup_{t\in(0,T)}\int_{\Omega}u^{\varepsilon}T_{k}(u^{\varepsilon})\,\mathrm{d}x+\int_{0}^{T}
		\int_{\Omega}\left|\nabla T_{k}(u^{\varepsilon})\right|^{p}\,\mathrm{d}x\mathrm{d}t\leq Mk,
	\end{equation}
	holds. Here $M=C_{3}+\|f^{\varepsilon}\|_{L^{1}(Q)}+\|u^{\varepsilon}_{0}\|_{L^{1}(\Omega)}$. Using Lemma\ref{lemmadinardo}, we compute that:
	\begin{equation}
		\|\left|u^{\varepsilon}\right|^{p-1}\|_{L^{m}(Q)}\leq C_{4}M^{\frac{(p-1)(N+p)}{p(N+1)-N}}\left|Q\right|^{\frac{1}{p'}\frac{N}{N+p}}\,m<\frac{p(N+1)-N}{N(p-1)},
	\end{equation}
	similarly,
	\begin{equation}
		\|\left|\nabla u^{\varepsilon}\right|^{p-1}\|_{L^{s}(Q)}\leq C_{4}M^{\frac{(p-1)(N+2)}{p(N+1)-N}}\,<\frac{p(N+1)-N}{(N+1)(p-1)},
	\end{equation}
	where
	\begin{equation}
		C_{4}=C_{4}(N,p,\|c\|_{L^{m}(Q)},\|f^{\varepsilon}\|_{L^{1}(Q)},\|u^{\varepsilon}_{0}\|_{L^{1}(\Omega)},\left|Q\right|).
	\end{equation}
	By using the technique of partition of the domain, we get
	\begin{equation}
		\begin{aligned}[b]
			\int_{\Omega}|u^{\varepsilon}|\,\mathrm{d}x
			\leq&\int_{{|u^{\varepsilon}\leq1|}\bigcup{|u^{\varepsilon}|>1}}\left|u^{\varepsilon}\right|\,\mathrm{d}x\\
			\leq&\int_{\Omega}1\,\mathrm{d}x+\int_{\Omega}u^{\varepsilon}T_{1}(u^{\varepsilon})\,\mathrm{d}x\\
			\leq&\left|\Omega\right|+\int_{\Omega}u^{\varepsilon}T_{1}(u^{\varepsilon})\,\mathrm{d}x,
		\end{aligned}
	\end{equation}
	then
	\begin{equation}\label{domainconverge}
		\|u^{\varepsilon}\|_{L^{\infty}((0/T);L^{1}(\Omega))}\leq\left|\Omega\right|+\sup_{t\in(0,T)}\int_{\Omega}u^{\varepsilon}T_{1}(u^{\varepsilon})\,x.
	\end{equation}
	From \eqref{c1}-\eqref{domainconverge} we obtain
	\begin{equation}\label{c5}
		\|u^{\varepsilon}\|_{L^{\infty}((0,T);L^{1}(\Omega))}\leq C_{5},
	\end{equation}
	here 
	\begin{equation}
		C_{5}=C_{5}(N,p,|Q|,\|c(x,t)\|_{L^{m}(Q)},\|u^{\varepsilon}_{0}\|_{L^{1}(\Omega)},\|f^{\varepsilon}\|_{L^{p'}(Q)}).
	\end{equation}
	And from \eqref{mk}, it is clear to see that for $\forall \varepsilon>0$
	\begin{equation}
		T_{k}(u^{\varepsilon}) \text{ is bounded in } L^{p}(0,T ;W^{1,p}_{\Gamma_{d}}(\Omega)).
	\end{equation}
\end{proof}
\begin{theorem}
	For any pointwise \(C^1\) function \(S\) such that \(S'\) has compact support, the term
	\begin{equation}\label{stbnd}
		\dfrac{\partial S(u^{\varepsilon})}{\partial t} \text{ is bounded in } L^{1}(Q)+L^{p'}(0,T;W^{-1,p'}(\Omega)).
	\end{equation}
\end{theorem}
\begin{proof}
	Similarly to the above, we need to test \eqref{approximateequation} by \(\partial S(u^{\varepsilon})\) to obtain
	\begin{equation}
		\begin{aligned}[b]
			\frac{\partial S(u^{\varepsilon})}{\partial t}&-\nabla\cdot(a(x,t,\nabla u^{\varepsilon})S(u^{\varepsilon})) +S''(u^{\varepsilon})a(x,t,\nabla u^{\varepsilon})\nabla u^{\varepsilon} \\
			-\nabla&\cdot(\Phi_{\varepsilon}(x,t,u^{\varepsilon})S(u^{\varepsilon}))+S'(u^{\varepsilon})\Phi_{\varepsilon}(x,t,u^{\varepsilon})\nabla u^{\varepsilon} \\
			&=f^{\varepsilon}S'(u^{\varepsilon}).
		\end{aligned}
	\end{equation}
	Then from \eqref{tklpbnd} and the property of \(a(x,t,\nabla u^{\varepsilon})\) we see
	\begin{equation}
		\begin{aligned}[b]
			-\nabla\cdot(a(x,t,\nabla u^{\varepsilon}))S'(u^{\varepsilon})+&S''{u^{\varepsilon}}a(x,t,\nabla u^{\varepsilon})\nabla u^{\varepsilon}-f^{\varepsilon}S'(u^{\varepsilon}) \\
			\text{ is bounded in } L^{1}(Q)+&L^{p'}(0,T;W^{-1,p'}(\Omega)), \text{ independently of } \varepsilon.
		\end{aligned}
	\end{equation}
	Since $\text{supp} S'\subset[-k,k]$, for \(0<\varepsilon<\dfrac{1}{k}\) we use Hölder's inequality to see
	\begin{equation}
		\begin{aligned}[b]
			&\int_{0}^{T}\int_{\Omega\bigcap\{|u^\varepsilon|\leq k\}}\nabla\cdot(\Phi_{\varepsilon}(x,t,u^{\varepsilon})S(u^{\varepsilon}))\varphi \,\mathrm{d}x\mathrm{d}t \\
			\leq&\left(\int_{0}^{T}\int_{\Omega\bigcap\{|u^\varepsilon|\leq k\}}|\Phi_{\varepsilon}(x,t,u^{\varepsilon})S(u^{\varepsilon})|^{p'}\,\mathrm{d}x\mathrm{d}t\right)^{\frac{1}{p'}}\left(\int_{0}^{T}\int_{\Omega\bigcap\{|u^\varepsilon|\leq k\}}|\nabla\varphi|^{p}\,\mathrm{d}x\mathrm{d}t\right)^{\frac{1}{p}},
		\end{aligned}
	\end{equation}
	and through \eqref{tklpbnd}
	\begin{equation}
		\begin{aligned}[b]
			\int_{0}^{T}\int_{\Omega\bigcap\{|u^\varepsilon|\leq k\}}|\Phi_{\varepsilon}(x,t,u^{\varepsilon})S(u^{\varepsilon})|^{p'}\,\mathrm{d}x\mathrm{d}t
			\leq&\int_{0}^{T}\int_{\Omega}|c(x,t)|^{p'}|T_{k}(u^{\varepsilon})|^{\gamma p'}|S(u^{\varepsilon})|^{p'}\,\mathrm{d}x\mathrm{d}t\\  
			\leq& k^{\gamma p'}\|c(x,t)\|^{p'}_{L^{p'}(Q)}\|S(u^{\varepsilon})\|^{p'}_{L^\infty(Q)}\\
			\leq& C_{6},
		\end{aligned}
	\end{equation}
	where 
	\[C_{6}=C_{6}(N,p,k,\|S\|_{L^{\infty}(Q)},\|c(x,t)\|_{L^{m}(Q)}). \]
	In the same way, combining \eqref{phiapproximate} and \eqref{tklpbnd} yields
	\begin{equation}
		\begin{aligned}[b]
			\int^{T}_{0}\int_{\Omega}|S''(u^{\varepsilon})\Phi_{\varepsilon}(x,t,u^{\varepsilon})||\nabla u^{\varepsilon}|\,\mathrm{d}x\mathrm{d}t
			=&\int^{T}_{0}\int_{\Omega}|S''(u^{\varepsilon})\Phi(x,t,T_{k}(u^{\varepsilon}))||\nabla T_{k}(u^{\varepsilon})|\,\mathrm{d}x\mathrm{d}t\\
			\leq&\int_{0}^{T}\int_{\Omega}|S''(u^{\varepsilon})|c(x,t)|T_{k}(u^{\varepsilon})|^{\gamma}|\nabla T_{k}(u^{\varepsilon})|\,\mathrm{d}x\mathrm{d}t\\
			\leq& k^{\gamma p'}\|c(x,t)\|_{L^{m}(Q)}\|S''\|_{L^{\infty}}\|\nabla T_{k}(u^{\varepsilon})\|_{L^{p}(Q)} \\
			\leq& C_{7},
		\end{aligned}
	\end{equation}
	where
	\[ C_{7}=C_{7}(N,p,k,\|S''\|_{L^{\infty}(Q)},\|c(x,t)\|_{L^{m}(Q)}).\]
	Then we obtain
	\begin{equation*}\\
		\nabla\cdot(\Phi_{\varepsilon}(x,t,u^{\varepsilon}))S'(u^{\varepsilon})-S''(u^{\varepsilon})\Phi_{\varepsilon}(x,t,u^{\varepsilon})\nabla u^{\varepsilon} \text{ is bounded in } L^{1}(Q)+L^{p'}(0,T;W^{-1,p'}(\Omega)).
	\end{equation*}
	It results
	\begin{equation}
		\dfrac{\partial S(u^{\varepsilon})}{\partial t} \text{ is bounded in } L^{1}(Q)+L^{p'}(0,T;W^{-1,p'}(\Omega)), \text{ independently of } \varepsilon.
	\end{equation}
\end{proof}
	From \eqref{tklpbnd}, \eqref{stbnd} and Aubin type Lemma\cite{Simon1987}, it follows that there exists a subsequence of \(u^\varepsilon\), still denoted by \(u^\varepsilon\) such that
	\begin{equation}\label{uuq}
		u^{\varepsilon}\rightarrow u\quad a.e.(x,t)\in Q.
	\end{equation}
	Also, from \eqref{tklpbnd}, there exists a subsequence of \(u^\varepsilon\), still denoted by \(u^\varepsilon\) such that
	\begin{equation}\label{tklp}
		T_{k}(u^{\varepsilon})\rightharpoonup T_{k}(u)\text{ in }L^{p}(0.T;W^{1,p}_{\Gamma_{d}}(\Omega)).
	\end{equation}
	Since $u^\varepsilon$ is uniformly bounded in $L^\infty(0,T;L^1(\Omega))$ with respect to $\varepsilon$, it follows that 
	$$u \in L^\infty(0,T;L^1(\Omega))$$.
\section{Existence of Renormalized Solution}
	We have already known that, for $\forall k,h>0$
	\begin{equation}
		\nabla T_{k}(u^{\varepsilon})\rightharpoonup \nabla T_{k}(u)\text{ in }\left(L^{p}(Q)\right)^{N},
	\end{equation}
	and
	\begin{equation}
		\nabla T_{2k}(u^{\varepsilon}-T_{h}(u^{\varepsilon}))\rightharpoonup \nabla T_{2k}(u-T_{h}(u))\text{ in }\left(L^{p}(Q)\right)^{N}.
	\end{equation}
	Next, we will conclude Theorem \eqref{tkstg} with the help of the function $\eta_{\mu,j}(u)$ defined as follows.
\begin{definition}
		Let $\eta_{\mu,j}(u)$ be a smooth approximation of the truncation function $T_k(u)$, introduced to handle non-zero data. This function satisfies the following properties:
		\begin{equation}
			\begin{aligned}[b]
				\eta_{\mu,j}(u)_{t}&=\frac{\partial}{\partial t}(T_{k}(u))_{\mu}-\mu e^{-\mu t}T_{k}(\psi_{j})\\
				&=\mu(T_{k}(u)-(T_{k}(u)_{\mu}))-e^{-\mu t}T_{k}(\psi_{j})\\
				&=\mu(T_{k}(u)-\eta_{\mu,j}(u)),
			\end{aligned}
		\end{equation}
		and
		\begin{equation}
			\eta_{\mu,j}(u)(0)=T_{k}(\psi_{j}).
		\end{equation}
		We have the following estimate:
		\begin{equation}
			\begin{aligned}[b]
				|\eta_{\mu,j}(u)|&\leq|(T_{k}(u))_{\mu}|+|\exp(-\mu t)T_{k}(\psi_{j})|\\
				&\leq k(1-\exp(-\mu t))+k\exp(-\mu t)\\
				&=k.
			\end{aligned}
		\end{equation}
		Moreover, it enjoys the following nice convergence properties:
		\begin{equation}
			\eta_{\mu,j}(u)\rightarrow T_{k}(u) \text{ strongly in }L^{p}(0,T;W^{1,p}_{\Gamma_{d}}(\Omega)),
		\end{equation}
		\begin{equation}
			\nabla\eta_{\mu,j}(u)\rightarrow\nabla T_{k}(u)\text{ strongly in }\left(L^{p}(Q)\right)^{N}.
		\end{equation}
\end{definition}
\begin{theorem}\label{tkstg}
		The sequence \(\nabla T_k(u^\epsilon)\) above satisfies
		\begin{equation}
			\nabla T_k(u^\epsilon) \to \nabla T_k(u) \text{ strongly in } (L^p(Q))^N.
		\end{equation}
\end{theorem}
\begin{proof}
		We employ a regularization technique\cite{Landes1981} to deal with the time derivative of the truncation function. For $\mu > 0$, we define a temporal regularization of the function $T_k(u)$:
	\begin{equation}
		\left(T_{k}(u)\right)_{\mu}(x,t)\coloneqq\mu\int^{t}_{-\infty}e^{\mu(s-t)}T_{k}(u(x,s))\,\mathrm{d}s
	\end{equation}
	For $s < 0$, naturally, we extend $T_{k}(u)$ to be $0$. It is obvious that
	\[
	\left(T_{k}(u)\right)_{\mu}\in L^{p}(0,T;W^{1,p}_{\Gamma_{d}}(\Omega))\bigcap L^{\infty}(Q)
	\]
	and
	\[
	\nabla\left(T_{k}(u)\right)_{\mu}\in\left(L^{p}(Q)\right)^{N}
	\]
	Then for $a.e.\,t\in(0,T)$ we have
	\begin{equation}
		\begin{aligned}[b]
			|\left(T_{k}(u)\right)_{\mu}|&\leq\mu\int^{t}_{-\infty}\left|e^{\mu(s-t)}T_{k}(u(x,s))\right|\,\mathrm{d}s\\
			&\leq k\int^{t}_{-\infty}e^{\mu(s-t)}\,(\mu(s-t))\\ 
			&\leq k(1-e^{-\mu t}),
		\end{aligned}
	\end{equation}
	and
	\begin{equation}
		\begin{aligned}[b]
			\frac{\partial \left(T_{k}(u)\right)_{\mu}}{\partial t}
			&=\mu\frac{\partial}{\partial t}\int^{t}_{-\infty}e^{\mu(s-t)}T_{k}(u(x,s))\,\mathrm{d}s\\
			=&\mu\frac{\partial}{\partial t}e^{-\mu t}\int^{t}_{-\infty}e^{\mu s}T_{k}(u(x,s))\,\mathrm{d}s\\
			=&-\mu^{2}e^{-\mu t}\int^{t}_{-\infty}e^{\mu s}T_{k}(u(x,s))\,s+\mu e^{-\mu t}e^{\mu t}T_{k}(u(x,t))\\
			=&\mu\left(T_{k}(u)-\left(T_{k}(u)\right)_{\mu}\right).
		\end{aligned}
	\end{equation}
	Then
	\begin{equation}
		\begin{aligned}[b]
			\nabla(T_k(u))_\mu &= \mu \int_{-\infty}^t e^{\mu(s-t)} \nabla T_k(u(x,s)) \,\mathrm{d}(s-t) \\
			&= \int_0^t \nabla T_k(u(x,s)) \, \mathrm{d}(e^{\mu(s-t)}) \\
			&= \nabla T_k(u(x,s)) e^{\mu(s-t)} \Big|_0^t - \int_0^t \mu e^{\mu(s-t)} (\nabla T_k(u) - (\nabla T_k(u))_\mu)\,\mathrm{d}t \\
			&= \nabla T_k(u(x,t)) - \int_0^t \mu e^{\mu(s-t)} (\nabla T_k(u) - (\nabla T_k(u))_\mu)\,\mathrm{d}t.
		\end{aligned}
	\end{equation}
	By a direct computation and the Lebesgue dominated convergence theorem, we obtain
	\begin{equation}
		\nabla\left(T_{k}(u)\right)_{\mu}\rightarrow\nabla T_{k}(u) \text{ in }(L^{p}(Q))^{N}.
	\end{equation}
	Take a sequence $\{\psi_j\} \subset C_0^\infty(\Omega)$ that converges strongly to $u_0$ in $L^1(\Omega)$ and satisfies:
	\begin{equation}
		\eta_{\mu,j}\equiv \left(T_{k}(u)\right)_{\mu}+e^{-\mu t}T_{k}(\psi_{j}),
	\end{equation}
	Fix a positive number $k > 0$ and let $h < k$. In \eqref{approximateequation}, choose the test function
	\begin{equation}
		\omega^\varepsilon = T_{2k}\bigl(u^\varepsilon - T_h(u^\varepsilon)\bigr) + T_k(u^\varepsilon) - \eta_{\mu,j}(u).
	\end{equation}
	Set $M = 4k + h$. Then, whenever $|u^\varepsilon| > M$, we have $\nabla u^\varepsilon = 0$. Consequently,
	\begin{equation}\label{innertof}
		\begin{aligned}[b]
			\int_{0}^{T}\langle\frac{\partial u^{\varepsilon}}{\partial t},\omega^{\varepsilon}\rangle\,\mathrm{d}t
			+&\int^{T}_{0}\int_{\Omega}a(x,t,u^{\varepsilon})\cdot\nabla\omega^{\varepsilon}\,\mathrm{d}x\mathrm{d}t
			+\int^{T}_{0}\int_{\Omega}\Phi_{\varepsilon}(x,t,u^{\varepsilon})\nabla\omega^{\varepsilon}\,\mathrm{d}x\mathrm{d}t \\
			=&\int^{T}_{0}\int_{\Omega}f^{\varepsilon}\omega^{\varepsilon}\,\mathrm{d}x\mathrm{d}t
		\end{aligned}
	\end{equation}
	Let $w(\varepsilon, \mu, j, h)$ denote a quantity such that
	\begin{equation}\label{w0}
		\lim_{h \to \infty} \lim_{j \to \infty} \lim_{\mu \to \infty} \lim_{\varepsilon \to 0} w(\varepsilon, \mu, j, h) = 0.
	\end{equation}
	First, for the first term in \eqref{innertof}, using the fact that $|\eta_{\mu,j}(u)| \leq k$, the test function $\omega^\varepsilon$ can be expressed as
	\begin{equation}
		\omega^\varepsilon = T_{h+k}\bigl(u^\varepsilon - \eta_{\mu,j}(u)\bigr) - T_{h-k}\bigl(u^\varepsilon - T_k(u^\varepsilon)\bigr).
	\end{equation}
	Then, by Lemma 2.1 in \cite{Porretta1999}, we obtain
	\begin{equation}
		\int_0^T \left\langle \frac{\partial u^\varepsilon}{\partial t}, \omega^\varepsilon \right\rangle \geq \omega(\varepsilon, \mu, j, h).
	\end{equation}
	It results
	\begin{equation}\label{iint3}
		\int^{T}_{0}\int_{\Omega}a(x,t,\nabla u^{\varepsilon})\nabla\omega^{\varepsilon}dxdt\leq\int^{T}_{0}\int_{\Omega}f^{\varepsilon}\omega^{\varepsilon}\,\mathrm{d}x\mathrm{d}t +\int^{T}_{0}\int_{\Omega}\Phi_{\varepsilon}(x,t,u^{\varepsilon})\nabla\omega^{\varepsilon}\,\mathrm{d}x\mathrm{d}t.
	\end{equation}
	The left-hand side above is considered on the two regions $\{|u^\varepsilon| > k\}$ and $\{|u^\varepsilon| \leq k\}$. From \eqref{as1} and \eqref{as2} it follows that
	\begin{align}\label{2partint}
		\begin{aligned}[b]
			&\int^{T}_{0}\int_{\Omega}a(x,t,u^{\varepsilon})\nabla\omega^{\varepsilon}\,\mathrm{d}x\mathrm{d}t \\
			=&\int^{T}_{0}\int_{\Omega}a(x,t,\nabla T_{M}(u^{\varepsilon}))\nabla T_{2k}(u^{\varepsilon}-T_{h}(u^{\varepsilon})+T_{k}(u^{\varepsilon})-\eta_{\mu,j}(u))\,\mathrm{d}x\mathrm{d}t\\ 
			\geq&\int^{T}_{0}\int_{\Omega}a(x,t,\nabla T_{k}(u^{\varepsilon}))\nabla(T_{k}(u^{\varepsilon})-\eta_{\mu,j}(u))\,\mathrm{d}x\mathrm{d}t
			-\iint_{\{|u^{\varepsilon}|>k\}}|a(x,t,\nabla T_{M}(u^{\varepsilon}))||\nabla\eta_{\mu,j}(u)|\,\mathrm{d}x\mathrm{d}t,
		\end{aligned}
	\end{align}
	Rearranging \eqref{iint3} and \eqref{2partint} yields
	\begin{equation}
		\begin{aligned}[b]
			&\int^{T}_{0}\int_{\Omega}a(x,t,\nabla T_{k}(u^{\varepsilon}))\nabla(T_{k}(u^{\varepsilon})-\eta_{\mu,j}(u))\,\mathrm{d}x\mathrm{d}t\\
			\leq&\int^{T}_{0}\int_{\Omega}
			a(x,t,\nabla u^{\varepsilon})\nabla\omega^{\varepsilon}\,\mathrm{d}x\mathrm{d}t+\iint_{\{|u^{\varepsilon}|>k\}}|a(x,t,\nabla T_{M}(u^{\varepsilon}))||\nabla\eta_{\mu,j}(u)|\,\mathrm{d}x\mathrm{d}t\\
			\leq&\int^{T}_{0}\int_{\Omega}f^{\varepsilon}\omega^{\varepsilon}\,\mathrm{d}x\mathrm{d}t+\int^{T}_{0}\int_{\Omega}
			\Phi_{\varepsilon}(x,t,u^{\varepsilon})\nabla\omega^{\varepsilon}\,\mathrm{d}x\mathrm{d}t+\iint_{\{|u^{\varepsilon}|>k\}}
			|a(x,t,T_{M}(u^{\varepsilon}))||\nabla\eta_{\mu,j}(u)|\,\mathrm{d}x\mathrm{d}t,
		\end{aligned}
	\end{equation}
	Then it turns into
	\begin{equation}\label{j1234}
		\begin{aligned}[b]
			&\int^{T}_{0}\int_{\Omega}[a(x,t,\nabla T_{k}(u^{\varepsilon}))-a(x,t,\nabla T_{k}(u))][\nabla T_{k}(u^{\varepsilon})-\nabla\eta_{\mu,j}(u)]\,\mathrm{d}x\mathrm{d}t\\
			=&\int^{T}_{0}\int_{\Omega}a(x,t,\nabla T_{k}(u^{\varepsilon}))(\nabla T_{k}(u^{\varepsilon})-\nabla\eta_{\mu,j}(u))\,\mathrm{d}x\mathrm{d}t \\
			&-\int^{T}_{0}\int_{\Omega}a(x,t,\nabla T_{k}(u))(\nabla T_{k}(u^{\varepsilon})-\nabla\eta_{\mu,j}(u))\,\mathrm{d}x\mathrm{d}t\\
			\leq&\int^{T}_{0}\int_{\Omega}f^{\varepsilon}T_{2k}(u^{\varepsilon}-T_{h}(u^{\varepsilon})+T_{k}(u^{\varepsilon}-\eta_{\mu,j}(u)))+w(\varepsilon,\mu,h,j)\\
			&+\int^{T}_{0}\int_{\Omega} \Phi_{\varepsilon}(x,t,u^{\varepsilon})\nabla T_{2k}(u^{\varepsilon}-T_{h}(u^{\varepsilon})+T_{k}(u^{\varepsilon})-\eta_{\mu,j}(u))\,\mathrm{d}x\mathrm{d}t\\
			&+\iint_{\{|u^{\varepsilon}|>k\}}
			|a(x,t,\nabla T_{M}(u^{\varepsilon}))||\nabla\eta_{\mu,j}(u)|\,\mathrm{d}x\mathrm{d}t\\
			&+\int^{T}_{0}\int_{\Omega}a(x,t,\nabla T_{k}(u))(\nabla T_{k}(u^{\varepsilon})-\nabla\eta_{\mu,j}(u))\,\mathrm{d}x\mathrm{d}t\\
			=&J_{1}+J_{2}+J_{3}+J_{4}+w(\varepsilon,\mu,j,h).
		\end{aligned}
	\end{equation}
	We will treat the cases of $J_1, J_2, J_3, J_4 $ separately in the limit  $\varepsilon \to 0$  using the Lebesgue dominated convergence theorem.
	\begin{equation}
		\begin{aligned}[b]
			\left|J_{1}\right|\leq&\int^{T}_{0}\int_{\Omega}\left|f^{\varepsilon}-f\right||T_{2k}(u^{\varepsilon}-T_{h}(u^{\varepsilon})+T_{k}(u^{\varepsilon})-\eta_{\mu,j}(u))|\,\mathrm{d}x\mathrm{d}t\\
			&+\int^{T}_{0}\int_{\Omega}\left|fT_{2k}(u^{\varepsilon}-T_{h}(u^{\varepsilon})+T_{k}(u^{\varepsilon})-\eta_{\mu,j}(u))\right|\,\mathrm{d}x\mathrm{d}t\\
			\leq&2k\int^{T}_{0}\int_{\Omega}\left|f^{\varepsilon}-f\right|\,\mathrm{d}x\mathrm{d}t+\int^{T}_{0}\int_{\Omega}\left|fT_{2k}(u^{\varepsilon}-T_{h}(u^{\varepsilon})+T_{k}(u^{\varepsilon})-\eta_{\mu,j}(u))\right|\,\mathrm{d}x\mathrm{d}t
		\end{aligned}
	\end{equation}
	It follows that
	\begin{equation}\label{j1}
		\begin{aligned}[b]
			&\lim_{j\rightarrow\infty}\lim_{h\rightarrow\infty}\lim_{\mu\rightarrow\infty}\lim_{\varepsilon\rightarrow 0}\left|J_{1}\right|\\
			\leq&\lim_{j\rightarrow\infty}\lim_{h\rightarrow\infty}\lim_{\mu\rightarrow\infty}\lim_{\varepsilon\rightarrow 0}\int^{T}_{0}\int_{\Omega}\left|fT_{2k}(u^{\varepsilon}-T_{h}(u^{\varepsilon})+T_{k}(u^{\varepsilon})-\eta_{\mu,j}(u))\right|\,\mathrm{d}x\mathrm{d}t\\
			=&0,
		\end{aligned}
	\end{equation}
	and
	\begin{equation}\label{j2}
		\begin{aligned}[b]
			&\lim_{j\rightarrow\infty}\lim_{h\rightarrow\infty}\lim_{\mu\rightarrow\infty}\lim_{\varepsilon\rightarrow 0}|J_2|\\
			=&\lim_{j\rightarrow\infty}\lim_{h\rightarrow\infty}\lim_{\mu\rightarrow\infty}\lim_{\varepsilon\rightarrow 0}\int^{T}_{0}\int_{\Omega}\Phi_{\varepsilon}(x,t,u^{\varepsilon})\nabla T_{2k}(u^{\varepsilon}-T_{h}(u^{\varepsilon})+T_{k}(u^{\varepsilon})-\eta_{\mu,j}(u))\,\mathrm{d}x\mathrm{d}t\\
			\leq&\lim_{j\rightarrow\infty}\lim_{h\rightarrow\infty}\lim_{\mu\rightarrow\infty}\lim_{\varepsilon\rightarrow 0}\int^{T}_{0}\int_{\Omega}c(x,t)\left|T_{k}(u^{\varepsilon})\right|^{\gamma}\nabla T_{2k}(u^{\varepsilon}-T_{h}(u^{\varepsilon})+T_{k}(u^{\varepsilon})-\eta_{\mu,j}(u))\,\mathrm{d}x\mathrm{d}t \\
			=&0.
		\end{aligned}
	\end{equation}
	Since
	\begin{equation}
		\left|a(x,t,\nabla T_{M}(u))\right|\text{ is bounded in }\left(L^{p'}(Q)\right)^{N},
	\end{equation}
	we obtain
	\begin{equation}\label{j3}
		\begin{aligned}[b]
			&\lim_{j\rightarrow\infty}\lim_{h\rightarrow\infty}\lim_{\mu\rightarrow\infty}\lim_{\varepsilon\rightarrow 0}|J_3|\\
			=&\lim_{j\rightarrow\infty}\lim_{h\rightarrow\infty}\lim_{\mu\rightarrow\infty}\lim_{\varepsilon\rightarrow  0}\iint_{\{\left|u^{\varepsilon}\right|>k\}}\left|a(x,t,\nabla  T_{M}(u^{\varepsilon}))\right|\left|\nabla\eta_{\mu,j}(u)\right|\,\mathrm{d}x\mathrm{d}t \\
			=&0.
		\end{aligned}
	\end{equation}
	Last but not the least,
	\begin{equation}\label{j4}
		\begin{aligned}[b]
			&\lim_{j\rightarrow\infty}\lim_{h\rightarrow\infty}\lim_{\mu\rightarrow\infty}\lim_{\varepsilon\rightarrow 0}|J_4|\\
			=&\lim_{j\rightarrow\infty}\lim_{h\rightarrow\infty}\lim_{\mu\rightarrow\infty}\lim_{\varepsilon\rightarrow 0}\int^{T}_{0}\int_{\Omega}a(x,t,\nabla T_{k}(u))(\nabla T_{k}(u^{\varepsilon})-\nabla\eta_{\mu,j}(u))\,\mathrm{d}x\mathrm{d}t \\
			=&0.
		\end{aligned}
	\end{equation}
	Combine \eqref{w0} and \eqref{j1234}-\eqref{j4} we see 
	\begin{equation}\label{iint0}
		\lim_{\varepsilon\rightarrow0}\int^{T}_{0}\int_{\Omega}[a(x,t,\nabla T_{k}(u^{\varepsilon}))-a(x,t,\nabla T_{k}(u))][\nabla T_{k}(u^{\varepsilon})-\nabla\eta_{\mu,j}(u)]\,\mathrm{d}x\mathrm{d}t=0,
	\end{equation}
	such that
	\begin{equation}\label{atkatk}
		a(x,t,\nabla T_{k}(u^{\varepsilon}))\nabla T_{k}(u^{\varepsilon})\rightarrow a(x,t,\nabla T_{k}(u))\nabla T_{k}(u),\quad a.e.(x,t)\in Q.
	\end{equation}
	Then from \eqref{uuq} and using Vitali's Theorem,
	\begin{equation}
		\int_{0}^{T}\int_{\Omega}a(x,t,\nabla T_{k}(u^{\varepsilon}))\nabla T_{k}(u^{\varepsilon})\,\mathrm{d}x\mathrm{d}t\rightarrow\int_{0}^{T}\int_{\Omega}a(x,t,\nabla T_{k}(u))\nabla T_{k}(u)\,\mathrm{d}x\mathrm{d}t,
	\end{equation}
	since
	\[
	a(x, t, \nabla T_k(u^\varepsilon)) \rightharpoonup a(x, t, \nabla T_k(u)) \text{ weakly in } (L^p(Q))^N
	\]
	or
	\[
	a(x, t, \nabla T_k(u^\varepsilon)) \to a(x, t, \nabla T_k(u)) \text{ a.e. in } Q.
	\]
	Set
	\[
	y^{\varepsilon}=\int_{0}^{T}\int_{\Omega}a(x,t,\nabla T_{k}(u^{\varepsilon}))\nabla T_{k}(u^{\varepsilon})\,\mathrm{d}x\mathrm{d}t
	\]
	and
	\[
	y=\int_{0}^{T}\int_{\Omega}a(x,t,\nabla T_{k}(u))\nabla T_{k}(u)\,\mathrm{d}x\mathrm{d}t
	\]
	Then by the non-negativity of \(a(x, t, \nabla T_k(u^\varepsilon))\) we conclude that
	\begin{equation*}
		\int_{0}^{T}\int_{\Omega}|y^{\varepsilon}-y|\,\mathrm{d}x\mathrm{d}t=2\int_{0}^{T}\int_{\Omega}(y-y^{\varepsilon})\,\mathrm{d}x\mathrm{d}t+\int_{0}^{T}\int_{\Omega}(y^{\varepsilon}-y)\,\mathrm{d}x\mathrm{d}t.
	\end{equation*}
	Using the Lebesgue dominated convergence theorem and \eqref{atkatk} we obtain
	\begin{equation*}
		y^{\varepsilon}\rightarrow y\text{ strongly in }L^{1}(Q),
	\end{equation*}
	as well as
	\begin{equation}\label{atkl1q}
		a(x,t,\nabla T_{k}(u^{\varepsilon}))\nabla T_{k}(u^{\varepsilon})\rightarrow a(x,t,\nabla T_{k}(u))\nabla T_{k}(u)\text{ strongly in }L^1(Q).
	\end{equation}
	Finally Lemma \ref{lemmaboccardo} shows that
	\begin{equation}
		\nabla T_{k}(u^{\varepsilon})\rightarrow\nabla T_{k}(u)\text{ strongly in }(L^p(Q))^N.
	\end{equation}
\end{proof}
\begin{lemma}\label{up0}
		Suppose \(u^\varepsilon\) is the solution of \eqref{approximateequation}, then for \(n\in N\) and \(\varepsilon > 0\), 
		\begin{equation}
			\lim_{n\rightarrow\infty}\limsup_{\varepsilon\rightarrow 0}\iint_{\{\left|u^{\varepsilon}\right|<n\}}\left|\nabla u^{\varepsilon}\right|^{p}dxdt=0
		\end{equation}
		is held
\end{lemma}
\begin{proof}
		We test the first equation of \eqref{approximateequation} by $\dfrac{T_{n}(u^{\varepsilon})}{n}$ to get
	\begin{equation}
		\begin{aligned}[b]
			&\frac{1}{n}\int^{T}_{0}\int_{\Omega}a(x,t,\nabla u^{\varepsilon})\nabla T_{n}(u^{\varepsilon})\,\mathrm{d}x\mathrm{d}t
			+\frac{1}{n}\int^{T}_{0}\int_{\Omega}\Phi_{\varepsilon}(x,t,u^{\varepsilon})\nabla T_{n}(u^{\varepsilon}\,\mathrm{d}x\mathrm{d}t\\
			=&\frac{1}{n}\int^{T}_{0}\int_{\Omega}f^{\varepsilon}T_{n}(u^{\varepsilon})\,\mathrm{d}x\mathrm{d}t
			-\frac{1}{n}\int^{T}_{0}\int_{\Omega}\frac{\partial u^{\varepsilon}}{\partial t}T_{n}(u^{\varepsilon})\,\mathrm{d}x\mathrm{d}t.
		\end{aligned}
	\end{equation}
	Rearranging yields
	\begin{equation}
		\begin{aligned}[b]
			&\frac{1}{2n}\int_{\Omega}\left|T_{n}(u^{\varepsilon})\right|^{2}\,\mathrm{d}x
			+\frac{\alpha}{n}\int_{0}^{T}\int_{\Omega}\left|\nabla T_{n}(u^{\varepsilon})\right|^{p}\,\mathrm{d}x\mathrm{d}t \\
			\leq&\frac{1}{n}\int^{T}_{0}\int_{\Omega}\left|f^{\varepsilon}T_{n}(u^{\varepsilon})\right|\,\mathrm{d}x\mathrm{d}t
			+\frac{1}{n}\int_{\Omega}\left|\Theta_{n}(u^{\varepsilon}_{0})\right|\,\mathrm{d}x
			+\frac{1}{n}\int^{T}_{0}\int_{\Omega}\left|\Phi_{\varepsilon}(x,t,u^{\varepsilon})\right|\left|\nabla T_{n}(u^{\varepsilon})\right|\,\mathrm{d}x\mathrm{d}t.
		\end{aligned}
	\end{equation}
	By \eqref{phiapproximate}, we set $0<\varepsilon<\frac{1}{n}$ to obtain
	\begin{equation}
		\begin{aligned}[b]
			&\frac{1}{n}\int_{0}^{T}\int_{\Omega}\left|\nabla T_{n}(u^{\varepsilon})\right|^{p}\,\mathrm{d}x\mathrm{d}t \\
			\leq&\frac{1}{n}\int^{T}_{0}\int_{\Omega}\left|f^{\varepsilon}T_{n}(u^{\varepsilon})\right|\,\mathrm{d}x\mathrm{d}t
			+\frac{1}{n}\int_{\Omega}\left|\Theta_{n}(u^{\varepsilon}_{0})\right|\,\mathrm{d}x
			+\frac{1}{n}\int^{T}_{0}\int_{\Omega}c(x,t)\left|T_{n}(u^{\varepsilon})\right|^{\gamma}\left|\nabla T_{n}(u^{\varepsilon})\right|\,\mathrm{d}x\mathrm{d}t.
		\end{aligned}
	\end{equation}
	Suppose $R>0$ and $E_{R}=\{(x,t)\in Q:\left|u^{\varepsilon}\right|>R\}$. In the case $R<n$ 
	\begin{equation}\label{1n3iint}
		\begin{aligned}[b]
			&\frac{1}{n}\int^{T}_{0}\int_{\Omega}c(x,t)\left|T_{n}(u^{\varepsilon})\right|^{\gamma}\left|\nabla T_{n}(u^{\varepsilon})\right|\,\mathrm{d}x\mathrm{d}t \\
			=&\frac{1}{n}\iint_{Q\setminus E_{R}}c(x,t)\left|T_{n}(u^{\varepsilon})\right|^{\gamma}\left|\nabla T_{n}(u^{\varepsilon})\right|\,\mathrm{d}x\mathrm{d}t+\frac{1}{n}\iint_{E_{R}}c(x,t)\left|T_{n}(u^{\varepsilon})\right|^{\gamma}\left|\nabla T_{n}(u^{\varepsilon})\right|\,\mathrm{d}x\mathrm{d}t\\
			=&I_{1}+I_{2}.
		\end{aligned}
	\end{equation}
	Since $T_{n}(u^{\varepsilon})\in L^{p}(0,T;W^{1,p}_{\Gamma_{d}}(\Omega))$, for \(I_1\) it results
	\begin{equation}\label{qer0}
		\lim_{n\rightarrow\infty}\limsup_{\varepsilon\rightarrow 0}\frac{1}{n}\iint_{Q\setminus E_{R}}c(x,t)\left|T_{n}(u^{\varepsilon})\right|^{\gamma}\left|\nabla T_{n}(u^{\varepsilon})\right|\,\mathrm{d}x\mathrm{d}t=0.
	\end{equation}
	For \(I_2\), using Hölder's inequality, the Gagliardo-Nirenberg inequality and Young's inequality, we obtain
	\begin{equation}\label{c8c9}
		\begin{aligned}[b]
			&\iint_{E_{R}}c(x,t)\left|T_{n}(u^{\varepsilon})\right|^{\gamma}\left|\nabla T_{n}(u^{\varepsilon})\right|\,\mathrm{d}x\mathrm{d}t\\
			\leq& C\|c(x,t)\|_{L^{m}(E_{R})}\left(\sup_{t\in(0,T)}\int_{\Omega}|T_{n}(u^{\varepsilon})|^{2}\,\mathrm{d}x\right)^{\frac{1}{2}}\left(
			\int^{T}_{0}\int_{\Omega}|\nabla T_{n}(u^{\varepsilon})|^{p}\,\mathrm{d}x\mathrm{d}t\right)^{\frac{N+1}{N+p}}\\
			\leq& C_{8}\|c(c,t)\|_{L^{m}(E_{R})}n^{\frac{1}{m}}\left(\sup_{t\in(0,T)}|T_{n}(u^{\varepsilon})|\,\mathrm{d}x\right)^{\frac{1}{m}}
			\left(\int_{0}^{T}\int_{\Omega}|\nabla T_{n}(u^{\varepsilon})|^{p}\,\mathrm{d}x\mathrm{d}t\right)^{\frac{N+1}{N+p}}\\
			\leq& C_{9}\|c(x,t)\|_{L^{m}(E_{R})}^{m}n+\frac{N+1}{N+p}\int_{0}^{T}\int_{\Omega}|\nabla T_{n}(u^{\varepsilon})|^{p}\,\mathrm{d}x\mathrm{d}t,
		\end{aligned}
	\end{equation}
	where $C_{8}=C_{8}(N,p)$ and $C_{9}=C_{9}(N,p)$. Combining \eqref{1n3iint} and \eqref{c8c9} we get
	\begin{equation}
		\begin{aligned}[b]
			&\frac{1}{n}\iint_{\{|u^{\varepsilon}|<n\}}|\nabla u^{\varepsilon}|^{p}\,\mathrm{d}x\mathrm{d}t \\
			\leq\frac{1}{n}&\iint_{Q\setminus E_{R}}c(x,t)|T_{R}(u^{\varepsilon})|^{\gamma}|\nabla T_{R}(u^{\varepsilon})|\,\mathrm{d}x\mathrm{d}t+\frac{1}{n}\int_{\Omega}|\Theta_{n}(u^{\varepsilon}_{0})|\,\mathrm{d}x\\
			+&\frac{1}{n}\int_{0}^{T}\int_{\Omega}|f^{\varepsilon}||T_{n}(u^{\varepsilon})|\,\mathrm{d}x\mathrm{d}t+C_{9}\|c(x,t)\|_{L^{m}(E_{R})}\\
			&+\frac{1}{n}\frac{N+1}{N+p}\int_{0}^{T}\int_{\Omega}|\nabla T_{n}(u^{\varepsilon})|^{p}\,\mathrm{d}x\mathrm{d}t.
		\end{aligned}
	\end{equation}
	This simplifies to
	\begin{equation}
		\begin{aligned}[b]
			\frac{1}{n}\iint_{\{|u^{\varepsilon}|<n\}}|\nabla u^{\varepsilon}|^{p}&\,\mathrm{d}x\mathrm{d}t
			\leq\frac{1}{n}\iint_{Q\setminus E_{R}}c(x,t)|T_{R}(u^{\varepsilon})|^{\gamma}|\nabla T_{R}(u^{\varepsilon})|\,\mathrm{d}x\mathrm{d}t\\
			+&\frac{1}{n}\int_{\Omega}|\Theta_{n}(u^{\varepsilon}_{0})|\,\mathrm{d}x
			+\frac{1}{n}\int_{0}^{T}\int_{\Omega}|f^{\varepsilon}||T_{n}(u^{\varepsilon})|\,\mathrm{d}x\mathrm{d}t+C_{9}\|c(x,t)\|_{L^{m}(E_{R})}.
		\end{aligned}
	\end{equation}
	Recalling \eqref{fst}-\eqref{ust} and \eqref{uuq}-\eqref{tklp} and using $u\in L^{\infty}(0,T;L^{1}(\Omega))$ to show
	\begin{equation}
		\lim_{n\rightarrow\infty}\limsup_{\varepsilon\rightarrow0}\frac{1}{n}\int_{0}^{T}\int_{\Omega}|f^{\varepsilon}||T_{n}(u^{\varepsilon})|\,\mathrm{d}x\mathrm{d}t=0,
	\end{equation}
	and
	\begin{equation}
		\lim_{n\rightarrow\infty}\limsup_{\varepsilon\rightarrow 0}\frac{1}{n}\int_{\Omega}|\Theta_{n}(u^{\varepsilon}_{0})|\,\mathrm{d}x=0.
	\end{equation}
	Thus,
	\begin{equation}
		\lim_{n\rightarrow\infty}\limsup_{\varepsilon\rightarrow 0}\frac{1}{n}\iint_{\{|u^{\varepsilon}|<n\}}|\nabla u^{\varepsilon}|^{p}\,\mathrm{d}x\mathrm{d}t\leq C_{10}\|c(x,t)\|_{L^{m}(E_{R})}.
	\end{equation}
	Since $u$ is finite almost everywhere in $Q$ and $c(x,t) \in L^m(Q)$, we have
	\begin{equation*}
		\lim_{R\rightarrow\infty}E_{R}=\lim_{R\rightarrow\infty}|\{(x,t)\in Q:|u^{\varepsilon}|>R\}|=0,
	\end{equation*}
	as well as
	\begin{equation*}
		\lim_{R\rightarrow\infty}\|c(x,t)\|_{L^{m}(E_{R})}=0.
	\end{equation*}
	Finally
	\begin{equation}
		\lim_{n\rightarrow\infty}\limsup_{\varepsilon\rightarrow 0}\frac{1}{n}\iint_{\{|u^{\varepsilon}|<n\}}|\nabla u^{\varepsilon}|^{p}\,\mathrm{d}x\mathrm{d}t=0.
	\end{equation}
\end{proof}
\begin{theorem}
		The limit $u$ of $u^\varepsilon$ in Lemma \ref{up0} is a solution to problem\eqref{equations}.
\end{theorem}
\begin{proof}
	For any $\varphi \in C^1(\overline{Q})$ with $\varphi(x, T) = 0$, and for any pointwise $C^1$ function $S \in W^{2,\infty}(\mathbb{R})$ with $\operatorname{supp} S' \subset [-M, M]$, taking $S'(u^\varepsilon)\varphi$ as a test function in \eqref{approximateequation} yields
	\begin{equation}
		\begin{aligned}[b]
			\int_{0}^{T}&\int_{\Omega}\frac{\partial S(u^{\varepsilon})}{\partial t}\varphi \,\mathrm{d}x\mathrm{d}t
			+\int_{0}^{T}\int_{\Omega}S'(u^{\varepsilon})a(x,t,\nabla u^{\varepsilon})\nabla\varphi \,\mathrm{d}x\mathrm{d}t
			+\int^{T}_{0}\int_{\Omega}S''(u^{\varepsilon})\varphi a(x,t,\nabla u^{\varepsilon})\nabla u^{\varepsilon}\,\mathrm{d}x\mathrm{d}t\\
			+&\int^{T}_{0}\int_{\Omega}S'(u^{\varepsilon})\Phi_{\varepsilon}(x,t,u^{\varepsilon})\nabla \varphi \,\mathrm{d}x\mathrm{d}t
			+\int^{T}_{0}\int_{\Omega}S''(u^{\varepsilon})\varphi\Phi_{\varepsilon}(x,t,u^{\varepsilon})\nabla u^{\varepsilon}\,\mathrm{d}x\mathrm{d}t \\
			&=\int^{T}_{0}\int_{\Omega}f^{\varepsilon}S'(u^{\varepsilon})\varphi \,\mathrm{d}x\mathrm{d}t
		\end{aligned}
	\end{equation}
	Since $S$ is bounded and continuous, from \eqref{uuq} we have
	\[
	S(u^{\varepsilon}) \to S(u) \quad \text{a.e. in } Q,
	\]
	and
	\[
	S(u^{\varepsilon}) \rightharpoonup S(u) \text{ weakly in } (L^\infty(Q))^*.
	\] 
	We conclude
	\begin{equation}\label{cvg1}
		\int^{T}_{0}\int_{\Omega}\frac{\partial S(u^{\varepsilon})}{\partial t}\,\mathrm{d}x\mathrm{d}t
		\rightarrow
		\int^{T}_{0}\int_{\Omega}\frac{\partial S(u)}{\partial t}\,\mathrm{d}x\mathrm{d}t.
	\end{equation}
	By the assumption $\text{supp}S'\subset[-M,M]$, we have
	\[
	S'(u^{\varepsilon})a(x,t,\nabla u^{\varepsilon})=S'((u^{\varepsilon}))a(x,t,\nabla T_{M}(u^{\varepsilon})),
	\]
	and
	\[
	S''(u^{\varepsilon})a(x,t,\nabla u^{\varepsilon})\nabla u^{\varepsilon}=S''(u^{\varepsilon})a(x,t,\nabla T_{M}(u^{\varepsilon}))\nabla T_{M}(u^{\varepsilon}).
	\]
	Using \eqref{uuq} and \eqref{atkatk} it results
	\begin{equation}\label{spmp}
		\begin{aligned}[b]
			&\int^{T}_{0}\int_{\Omega}\left|S'(u^{\varepsilon})a(x,t,\nabla T_{M}(u^{\varepsilon}))-S'(u)a(x,t,\nabla T_{M}(u))\right|^{p'}\,\mathrm{d}x\mathrm{d}t\\
			\leq& 2^{p}\int_{0}^{T}\int_{\Omega}\left|S'(u^{\varepsilon})\right|^{p'}[a(x,t,\nabla T_{M}(u^{\varepsilon}))-a(x,t,\nabla T_{M}(u))]^{p'}\,\mathrm{d}x\mathrm{d}t\\
			&+2^{p}\int^{T}_{0}\int_{\Omega}\left|S'(u^{\varepsilon})-S'(u)\right|^{p'}(a(x,t,\nabla T_{M}(u)))^{p'}\,\mathrm{d}x\mathrm{d}t\\
			\leq&2^{p}M^{p'}\int^{T}_{0}\int_{\Omega}[a(x,t,\nabla T_{M}(u^{\varepsilon}))-a(x,t,\nabla T_{M}(u))]^{p'}\,\mathrm{d}x\mathrm{d}t\\
			&+2^{p}\left|S'(u^{\varepsilon})-S'(u)\right|^{p'}_{L^{\infty}(Q)}\int^{T}_{0}\int_{\Omega}(a(x,t,\nabla T_{M}(u)))^{p'}\,\mathrm{d}x\mathrm{d}t.
		\end{aligned}
	\end{equation}
	Since \eqref{spmp} tends to \(0\) as $\varepsilon\rightarrow0$,
	\begin{equation}
		S'(u^{\varepsilon})a(x,t,\nabla T_{M}(u^{\varepsilon}))\rightarrow S'(u)a(x,t,\nabla T_{M}(u))\text{ in }(L^{p'}(Q))^{N}.
	\end{equation}
	Similarly,
	$$
	\int^{T}_{0}\int_{\Omega}\frac{\partial S'(u^{\varepsilon})}{\partial t}\,\mathrm{d}x\mathrm{d}t\rightarrow\int^{T}_{0}\int_{\Omega}\frac{\partial S'(u)}{\partial t}\,\mathrm{d}x\mathrm{d}t
	$$
	Then
	\[
	S''(u^{\varepsilon}) \overset{*}{\rightharpoonup} S''(u^*) \quad \text{ weakly-* in } L^\infty(Q).
	\] 
	And
	\begin{equation*}
		\iint_{Q}S''(u^{\varepsilon})\nabla T_{M}(u^{\varepsilon})\varphi\rightarrow\iint_{Q}S''(u^{*})\nabla T_{M}(u)\varphi=\iint_{Q}S''(u)\nabla T_{M}(u^{*})\varphi
	\end{equation*}
	is held. This deduces $u^{*}=u$ and
		\[
	S''(u^{\varepsilon}) \overset{*}{\rightharpoonup} S''(u) \quad \text{ weakly-* in } L^\infty(Q).
	\] 
	A direct computation yields
	\begin{equation}\label{2ps}
		\begin{aligned}[b]
		&\int_{0}^{T}\int_{\Omega}\left|S''(u^{\varepsilon})a(x,t,\nabla T_{M}u^{\varepsilon})\nabla T_{M}(u^{\varepsilon})-S''(u)a(x,t,\nabla T_{M}(u))\right|\,\mathrm{d}x\mathrm{d}t\\
		\leq&2^{p}\int^{T}_{0}\int_{\Omega}|S''(u^{\varepsilon})|[a(x,t,\nabla T_{M}(u^{\varepsilon}))\nabla T_{M}(u^{\varepsilon})-a(x,t,\nabla T_{M}(u))\nabla T_{M}(u)]\,\mathrm{d}x\mathrm{d}t\\
		&+2^{p}\int^{T}_{0}\int_{\Omega}|S''(u^{\varepsilon})-S''(u)|a(x,t,\nabla T_{M}(u))\nabla T_{M}(u)\,\mathrm{d}x\mathrm{d}t\\
		\leq&2^{p}\|S''(u^{\varepsilon})\|_{L^{\infty}(Q)}\int_{0}^{T}\int_{\Omega}[a(x,t,\nabla T_{M}(u^{\varepsilon}))\nabla T_{M}(u^{\varepsilon})-a(x,t,\nabla T_{M}(u))\nabla T_{M}(u)]\,\mathrm{d}x\mathrm{d}t\\
		&+2^{p}\|S''(u^{\varepsilon})-S''(u)\|_{L^{\infty}(Q)}\int^{T}_{0}\int_{\Omega}a(x,t,\nabla T_{M}(u))\nabla T_{M}(u)\,\mathrm{d}x\mathrm{d}t.
		\end{aligned}
	\end{equation}
	Since \eqref{2ps} tends to \(0\) as $\varepsilon\rightarrow0$,
	\begin{equation*}
		S''(u^{\varepsilon})a(x,t,\nabla T_{M}(u^{\varepsilon}))\nabla T_{M}(u^{\varepsilon})\rightarrow S''(u)a(x,t,\nabla T_{M}(u))\nabla T_{M}(u) \text{ in }L^{1}(Q).
	\end{equation*}
	From
	\[
	S'(u)a(x,t,\nabla T_{M}(u))=S'(u)a(x,t,\nabla u)
	\]
	and
	\[
	S''(u)a(x,t,\nabla T_{M}(u))\nabla T_{M}(u)=S''(u)a(x,t,\nabla u)\nabla u,
	\]
	we obtain
	\begin{equation}\label{cvg2}
		S'(u^{\varepsilon})a(x,t,\nabla u^{\varepsilon})\rightarrow S'(u)a(x,t,\nabla u)\text{ in }\left(L^{p'}(Q)\right)^{N}
	\end{equation}
	and
	\begin{equation}\label{cvg3}
		S''(u^{\varepsilon})a(x,t,\nabla u^{\varepsilon})\nabla u^{\varepsilon}\rightarrow S''(u)a(x,t,\nabla u)\nabla u\text{ in }L^{1}(Q).
	\end{equation}
	Similarly,
	\[
	S'(u^{\varepsilon})\Phi_{\varepsilon}(x,t,u^{\varepsilon})=S'(u^{\varepsilon})\Phi_{\varepsilon}(x,t,T_{M}(u^{\varepsilon}))\,\, a.e.(x,t)\in Q
	\]
	and
	\[
	S''(u^{\varepsilon})\Phi_{\varepsilon}(x,t,u^{\varepsilon})\nabla u^{\varepsilon}=S''(u^{\varepsilon})\Phi_{\varepsilon}(x,t,T_{M}(u^{\varepsilon}))\nabla T_{M}(u^{\varepsilon})\,\, a.e.(x,t)\in Q.
	\]
	From \eqref{phiapproximate} we know that
	\begin{equation}
		|S'(u^{\varepsilon})\Phi_{\varepsilon}(x,t,T_{M}(u^{\varepsilon}))\leq M^{\gamma}c(x,t)|S'(u^{\varepsilon})|
	\end{equation}
	Then by \eqref{uuq}
	\begin{equation}\label{cvg4}
		S'(u^{\varepsilon})\Phi_{\varepsilon}(x,t,u^{\varepsilon})\rightarrow S'(u)\Phi(x,t,u) \text{ in }\left(L^{p'}(Q)\right)^{N}.
	\end{equation}
	In the same way
	\begin{equation*}
	S''(u^{\varepsilon})\Phi_{\varepsilon}(x,t,u^{\varepsilon})\rightarrow S''(u)\Phi(x,t,u)\,\, a.e.(x,t)\in Q,
	\end{equation*}
	From \eqref{tklp} we get
	\begin{equation}\label{cvg5}
		S''(u^{\varepsilon})\Phi_{\varepsilon}(x,t,u^{\varepsilon})\nabla u^{\varepsilon}\rightarrow S''(u)\Phi(x,t,u)\nabla u.
	\end{equation}
	Since 
	\begin{equation}
		\begin{aligned}[b]
			&\int^{T}_{0}\int_{\Omega}|f^{\varepsilon}S'(u^{\varepsilon}\varphi)-fS'(u)\varphi|\,\mathrm{d}x\mathrm{d}t \\
			\leq& \int^{T}_{0}\int_{\Omega}f^{\varepsilon}|S'(u^{\varepsilon})-S'(u)|\varphi \,\mathrm{d}x\mathrm{d}t
			+\int^{T}_{0}\int_{\Omega}|f^{\varepsilon}|S'(u)\varphi \, \mathrm{d}x\mathrm{d}t\\
			\leq& \|S'(u^{\varepsilon})-S'(u)\|_{L^{\infty}}\|\varphi\|_{L^{\infty}}\int^{T}_{0}\int_{\Omega}f^{\varepsilon}\,\mathrm{d}x\mathrm{d}t
			+\|S'(u)\|_{L^{\infty}}\|\varphi\|_{L^{\infty}}\int^{T}_{0}\int_{\Omega}|f^{\varepsilon}-f|\,\mathrm{d}x\mathrm{d}t,
		\end{aligned}
	\end{equation}
	using \eqref{fst}and the property of \(S\) we obtain
	\begin{align}\label{cvg6}
		f^{\varepsilon}S'(u^{\varepsilon})\rightarrow fS'(u)\text{ in }L^{1}(Q).
	\end{align}
	Now Combining \eqref{cvg1}-\eqref{cvg6}
	\begin{equation}
		\begin{aligned}[b]
			-\int_{\Omega}\varphi(&x,0)S(u_0)\,\mathrm{d}x
			+\int^0_T \int_{\Omega}S(u)\frac{\partial \varphi}{\partial t}\,\mathrm{d}x\mathrm{d}t
			+\int^0_T \int_{\Omega}S'a(x,t,\nabla u)\nabla \varphi \,\mathrm{d}x\mathrm{d}t \\
			+\int^0_T& \int_{\Omega} S''(u)\varphi a(x,t,\nabla u)\nabla u \,\mathrm{d}x\mathrm{d}t 
			+\int^0_T \int_{\Omega}\int_{\Omega}S'(u)\Phi(x,t,u)\nabla \varphi \,\mathrm{d}x\mathrm{d}t \\ 
			+&\int^0_T \int_{\Omega}S''(u)\varphi\Phi(x,t,u)\nabla u \,\mathrm{d}x\mathrm{d}t \\
			&=\int^0_T \int_{\Omega}fS'(u)\varphi \,\mathrm{d}x\mathrm{d}t.
		\end{aligned}
	\end{equation}
	This completes the proof that \(u\) is a renormalized solution to Problem \eqref{equations}.
\end{proof}
\begin{remark}
	If $u$ is a renormalized solution, then $u$ satisfies the following conditions:
		\begin{equation*}
			\lim_{n\rightarrow\infty}\frac{1}{n}\iint_{\{(x,t)\in Q:|u|<n\}}|\Phi(x,t,u)||\nabla u|\,\mathrm{d}x\mathrm{d}t=0.
		\end{equation*}
\end{remark}
\begin{proof}
	From \eqref{as5} we get
	\begin{equation*}
		\iint_{Q}|\Phi(x,t,u)||\nabla u|dxdt\leq\iint_{Q}c(x,t)|u|^{\gamma}|\nabla T_{n}(u)|\,\mathrm{d}x\mathrm{d}t.
	\end{equation*}
	Using Hölder's inequality and the Gagliardo-Nirenberg inequality we obtain
	\begin{equation}
		\begin{aligned}
			&\iint_{Q}c(x,t)|u|^{\gamma}|\nabla u|\,\mathrm{d}x\mathrm{d}t\\
			\leq&\left(\iint_{Q}c^{r}(x,t)\right)^{\frac{1}{r}}\left(\iint_{Q}|T_{n}(u)|^{\frac{(N+2)p}{N}}\right)^{N(p-1)}p(N+p)
			\left(\iint_{Q}|\nabla T_{n}(u)|^{p}\right)^{\frac{1}{p}}\\
			\leq& n^{\frac{1}{r}}C\|c(x,t)\|_{L^{r}(Q)}\|u\|^{\frac{1}{r}}_{L^{\infty}(0,T;L^{1}(\Omega))}\left(\iint_{Q}|\nabla u|^{p}\,\mathrm{d}x\mathrm{d}t\right)^{\frac{N+1}{N+p}},
		\end{aligned}
	\end{equation}
	where $1-\frac{1}{r}=\frac{N+1}{N+p}$. Recalling assumption \eqref{as1} and \eqref{iintau0}
	\begin{equation*}
		\lim_{n\rightarrow\infty}\frac{1}{n}\iint_{Q}c(x,t)|T_{n}(u)|^{\gamma}|\nabla T_{n}(u)|\,\mathrm{d}x\mathrm{d}t=0.
	\end{equation*}
	Thus
	\begin{equation*}
		\lim_{n\rightarrow\infty}\frac{1}{n}\iint_{\{(x,t)\in Q:|u|<n\}}|\Phi(x,t,u)||\nabla u|\,\mathrm{d}x\mathrm{d}t=0.
	\end{equation*}
\end{proof}	
\section{Uniqueness of Renormalized Solution}
\begin{theorem}
	If
	\begin{equation}
		|\Phi(x,t,r_{1})-\Phi(x,t,r_{2})|\leq|r_{1}-r_{2}|,\quad r_{1},r_{2}\in R,\quad(x,t)\in Q,
	\end{equation}
	then \eqref{equations} has a unique renormalized solution.
\end{theorem}
\begin{definition}
	For $s > 0$, $\sigma > 0$, define $T_s^\sigma(u)$ as follows:
	\begin{equation}
		T_s^\sigma(0) = 0, 
	\end{equation}
	\begin{equation}
		T_s^\sigma(r) =
		\begin{cases}
			1, & |r| < s, \\[4pt]
			\dfrac{1}{\sigma}(s + \sigma - |r|), & s \leq |r| \leq s + \sigma, \\[10pt]
			0, & |r| > s + \sigma.
		\end{cases}
	\end{equation}
\end{definition}
\begin{proof}
	Suppose \(u\) and \(v\) are two renormalized solution of problem \eqref{equations} with the same initial value \(f\) and \(u_0\).
	Recalling \eqref{longeqv}, we set
	\[
	S=T^{\sigma}_{s}
	\]
	and
	\[
	\varphi=\frac{1}{k}T_{k}(T^{\sigma}_{s}(u)-T^{\sigma}_{s}(v)).
	\]
	Substitute $u$ and $v$ into the equality respectively, and then subtract these two we obtain
	\begin{equation}\label{i012345}
		I_{0}+I_{1}+I_{2}+I_{3}+I_{4}=I_{5}.
	\end{equation}
	Here
	\begin{equation}
		\begin{aligned}[b]
			I_{0}=\frac{1}{k}\int^{t}_{0}\left\langle \frac{\partial (T^{\sigma}_{s}(u)-T^{\sigma}_{s}(v))}{\partial \tau}, T_{k}\left(T^{\sigma}_{s}(u)-T^{\sigma}_{s}(v)\right) \right\rangle \,\mathrm{d}\tau
		\end{aligned}
	\end{equation}
	\begin{equation}
		\begin{aligned}[b]
			I_{1}=\frac{1}{k}\int^{t}_{0}\int_{\Omega}\bigl[(T^{\sigma}_{s})'(u)a(x,\tau,\nabla u)-(T^{\sigma}_{s})'(v)a(x,\tau,\nabla v)\bigr] \nabla T_{k}\bigl(T^{\sigma}_{s}(u)-T^{\sigma}_{s}(v)\bigr) \,\mathrm{d}x \mathrm{d}\tau
		\end{aligned}
	\end{equation}
	\begin{equation}
		\begin{aligned}[b]
			I_{2}=&\frac{1}{k}\int_{0}^{t}\int_{\Omega}(T^{\sigma}_{s})''(u)a(x,\tau,\nabla u)\nabla u \, T_{k}\bigl(T^{\sigma}_{s}(u)-T_{s}^{\sigma}(v)\bigr) \,\mathrm{d}x \mathrm{d}\tau \\
			&- \frac{1}{k}\int_{0}^{t}\int_{\Omega}(T^{\sigma}_{s})''(v)a(x,\tau,\nabla v)\nabla v \, T_{k}\bigl(T^{\sigma}_{s}(u)-T_{s}^{\sigma}(v)\bigr) \,\mathrm{d}x \mathrm{d}\tau,
		\end{aligned}
	\end{equation}
	\begin{equation}
		\begin{aligned}[b]
			I_{3}=\frac{1}{k}\int_{0}^{t}\int_{\Omega}\bigl[(T^{\sigma}_{s})'(u)\Phi(x,\tau,u)-(T^{\sigma}_{s})'(v)\Phi(x,\tau,v)\bigr] \nabla T_{k}\bigl(T^{\sigma}_{s}(u)-T^{\sigma}_{s}(v)\bigr) \,\mathrm{d}x \mathrm{d}\tau,
		\end{aligned}
	\end{equation}
	\begin{equation}
		\begin{aligned}[b]
			I_{4}=&\frac{1}{k}\int_{0}^{t}\int_{\Omega}(T^{\sigma}_{s})''(u)\Phi(x,\tau,\nabla u)\nabla u \, T_{k}\bigl(T^{\sigma}_{s}(u)-T_{s}^{\sigma}(v)\bigr) \,\mathrm{d}x \mathrm{d}\tau\\
			&- \frac{1}{k}\int_{0}^{t}\int_{\Omega}(T^{\sigma}_{s})''(v)\Phi(x,\tau,\nabla v)\nabla v \, T_{k}\bigl(T^{\sigma}_{s}(u)-T_{s}^{\sigma}(v)\bigr) \,\mathrm{d}x \mathrm{d}\tau,
		\end{aligned}
	\end{equation}
	\begin{equation}
		\begin{aligned}[b]
			I_{5}=\frac{1}{k}\int^{t}_{0}\int_{\Omega}f\bigl[(T^{\sigma}_{s})'(u)-(T^{\sigma}_{s})'(v)\bigr] T_{k}\bigl(T^{\sigma}_{s}(u)-T_{s}^{\sigma}(v)\bigr) \,\mathrm{d}x \mathrm{d}\tau.
		\end{aligned}
	\end{equation}
	For any $t \in (0, T)$, the notation $\langle\cdot, \cdot\rangle$ denotes the duality pairing between $L^1(\Omega) + W^{-1,p'}(\Omega)$ and $L^\infty(\Omega) \cap W_{\Gamma_d}^{1,p}(\Omega)$.
	Taking the limits in \eqref{i012345} as $\sigma \to 0$, $k \to 0$, and $s \to \infty$, it is clear that for any $t \in (0, T)$,
	\begin{equation}\label{tsstg}
		T_s^\sigma(u) \to T_s(u) \quad \text{a.e. in } Q \text{ and strongly in } L^p(0,T; W_{\Gamma_d}^{1,p}(\Omega)),
	\end{equation}
	\begin{equation}\label{tustg}
		(T_s^\sigma)'(u) \to \chi_{\{|u| \leq s\}} \quad \text{a.e. in } Q \text{ and strongly in } L^q(\Omega \times (0,t)).
	\end{equation}
	For $I_0$,
	\[
	I_0 = \frac{1}{k} \int_0^t \left\langle \frac{T_s^\sigma(u) - T_s^\sigma(v)}{\partial \tau}, T_k(T_s^\sigma(u) - T_s^\sigma(v)) \right\rangle \,\mathrm{d}\tau
	\]
	\[
	= \frac{1}{k} \int_\Omega \Theta_k(T_s^\sigma(u) - T_s^\sigma(v))(x,t) \, \mathrm{d}x - \frac{1}{k} \int_\Omega \Theta_k(T_s^\sigma(u) - T_s^\sigma(v))(x,0) \, \mathrm{d}x.
	\] 
	Then
	\begin{equation}\label{tsint0}
		\lim_{k \to 0} \lim_{\sigma \to 0} I_0 = \int_{\Omega} \bigl| T_s(u)(t) - T_s(v)(t) \bigr| \, \mathrm{d}x.
	\end{equation}
	For $I_1$, by \eqref{tsstg} and \eqref{tustg}, together with \eqref{as1}, and since $\operatorname{supp} (T_s^\sigma)' \subset [-s-\sigma, s+\sigma]$, for any $t \in (0, T)$ we have
	\[
	\lim_{\sigma \to 0} I_1 = \frac{1}{k} \int_0^t \int_\Omega \bigl[ \chi_{\{|u| \leq s\}} a(x, \tau, \nabla u) - \chi_{\{|v| \leq s\}} a(x, \tau, \nabla v) \bigr] \nabla T_k\bigl(T_s(u) - T_s(v)\bigr) \, \mathrm{d}x\mathrm{d}\tau
	\]
	\[
	= \frac{1}{k} \int_0^t \int_\Omega \bigl[ a(x, \tau, \nabla T_s(u)) - a(x, \tau, \nabla T_s(v)) \bigr] \nabla T_k\bigl(T_s(u) - T_s(v)\bigr) \, \mathrm{d}x  \mathrm{d}\tau \geq 0.
	\] 	
	Then
	\[
	\lim_{\sigma \to 0} \liminf_{k \to 0}  I_1 \geq 0.
	\] 
	For $I_5$,
	\[
	\lim_{\sigma \to 0} I_5 = \frac{1}{k} \int_0^t \int_\Omega f \times \bigl( \chi_{\{|u| \leq s\}} - \chi_{\{|v| \leq s\}} \bigr) T_k\bigl(T_s(u) - T_s(v)\bigr) \, \mathrm{d}x \mathrm{d}\tau,
	\]
	\[
	\lim_{k \to 0} \lim_{\sigma \to 0} I_5 = \int_0^t \int_\Omega f \times \bigl( \chi_{\{|u| \leq s\}} - \chi_{\{|v| \leq s\}} \bigr) \operatorname{sgn}(u - v) \, \mathrm{d}x \mathrm{d}\tau.
	\] 
	where $\operatorname{sgn}(r) = \dfrac{r}{|r|}$ for any $r \neq 0$, and $\operatorname{sgn}(0) = 0$. Since $u$ and $v$ are finite almost everywhere in $\Omega \times (0, T)$ and $f \in L^1(Q)$, by the Lebesgue dominated convergence theorem we obtain
	\[
	\lim_{s \to \infty} \lim_{k \to 0} \lim_{\sigma \to 0} I_5 = 0.
	\]
	Next we show that for any $t \in (0, T)$,
	\begin{equation}\label{m1}
		|I_2| + |I_4| \leq \frac{M_1}{\sigma} \Gamma(u, v, s, \sigma),
	\end{equation}
	where $M_1$ is a constant independent of $s$, $k$, and $\sigma$. Using the definition of $T_s^\sigma$ and the fact that $\nabla u = 0$ a.e. on $\{ (x, t) : |u(x, t)| = r \}$ for any $r \in \mathbb{R}$, together with $a(x, t, \xi) \cdot \xi \geq 0$, we have
	\begin{equation}\label{i2}
		\begin{aligned}[b]
			|I_{2}|\leq&\frac{1}{k}\int^{t}_{0}\int_{\Omega}(T^{\sigma}_{s})''(u)a(x,t,\nabla u)\nabla u T_{k}(T_{s}^{\sigma}(u)-T_{s}^{\sigma}(v))\,\mathrm{d}x\mathrm{d}t\\
			&+\frac{1}{k}\int^{t}_{0}\int_{\Omega}(T^{\sigma}_{s})''(u)
			a(x,t,\nabla v)\nabla v T_{k}(T^{\sigma}_{s}(u)-T^{\sigma}_{s}(v))\,\mathrm{d}x\mathrm{d}t\\
			\leq&\int^{t}_{0}\int_{\Omega}|(T^{\sigma}_{s})''(u)a(x,t,\nabla u)\nabla u|\,\mathrm{d}x\mathrm{d}t
			+\int^{t}_{0}\int_{\Omega}|(T^{\sigma}_{s})''(v)a(x,t,\nabla v)\nabla v|\,\mathrm{d}x\mathrm{d}t\\
			\leq&\frac{1}{\sigma}\left[\iint_{\{s\leq|u|\leq s+\sigma\}}a(x,t,\nabla u)\nabla u\,\mathrm{d}x\mathrm{d}t
			+\iint_{\{s\leq|v|\leq s+\sigma\}}a(x,t,\nabla v)\nabla v\,\mathrm{d}x\mathrm{d}t\right]
		\end{aligned}
	\end{equation}
	Analogously, 
	\begin{equation}\label{i4}
		\begin{aligned}[b]
			|I_{4}|\leq&\frac{1}{k}\int^{t}_{0}\int_{\Omega}|(T^{\sigma}_{s})''(u)\Phi(x,t,u)\nabla u T_{k}(T^{\sigma}_{s}(u)-T^{\sigma}_{s}(v))|\,\mathrm{d}x\mathrm{d}t\\
			&+\frac{1}{k}\int^{t}_{0}\int_{\Omega}|(T^{\sigma}_{s})''(v)\Phi(x,t,v)\nabla vT_{k}(T^{\sigma}_{s}(u)-T^{\sigma}_{s}(v))\,\mathrm{d}x\mathrm{d}t\\
			\leq&\int^{t}_{0}\int_{\Omega}|(T^{\sigma}_{s})''(u)\Phi(x,t,u)\nabla u|\,\mathrm{d}x\mathrm{d}t
			+\int^{t}_{0}\int_{\Omega}|(T^{\sigma}_{s})''(v)\Phi(x,t,v)\nabla v|\,\mathrm{d}x\mathrm{d}t\\
			\leq& \frac{1}{\sigma}\left(\iint_{\{s\leq|u|\leq s+\sigma\}}|\Phi(x,t,u)\nabla u|\,\mathrm{d}x\mathrm{d}t
			+\iint_{\{s\leq|v|\leq s+\sigma\}}|\Phi(x,t,v)\nabla v|\,\mathrm{d}x\mathrm{d}t\right).
		\end{aligned}
	\end{equation}
	Therefore, from \eqref{i2} and \eqref{i4} we obtain \eqref{m1}. For $I_5$, we have
	\begin{equation}
		\begin{aligned}[b]
			\limsup_{\sigma\rightarrow 0}|I_{3}|=&|\frac{1}{k}\int^{t}_{0}\int_{\Omega}|\left[\chi_{\{|u|\leq s\}}\Phi(x,t,u)-\chi_{\{|v|\leq s\}}\Phi(x,t,v)\right]\nabla T_{k}(T_{s}(u)-T_{s}(v))\,\mathrm{d}x\mathrm{d}t|\\
			\leq&\frac{1}{k}\iint_{\{|u|\leq s\}\bigcap\{|v|>s\}}|\Phi(x,t,u)||\nabla T_{k}(T_{s}(u)-s\,\text{sgn}(v))|\,\mathrm{d}x\mathrm{d}t\\
			&+\frac{1}{k}\iint_{\{|v|\leq s\}\bigcap\{|u|>s\}}|\Phi(x,t,v)||\nabla T_{k}(T_{s}(v)-s\,\text{sgn}(u))|\,\mathrm{d}x\mathrm{d}t\\
			&+\frac{1}{k}\iint_{\{|v|\leq s\}\bigcap\{|u|\leq s\}}|\Phi(x,t,u)-\Phi(x,t,v)||\nabla T_{k}(u-v)|\,\mathrm{d}x\mathrm{d}t \\
			=&I_{3}^{1}+I_{3}^{2}+I_{3}^{3}.
		\end{aligned}
	\end{equation}
	Estimating $I_3^1$ and $I_3^2$, we have
	\begin{equation}
		\begin{aligned}[b]
			I_{3}^{1}&\leq\frac{1}{k}\iint_{\{|u|\leq s\}\bigcap\{|v|>s\}\bigcap\{|u-s sgn(v)|<k\}}|\Phi(x,t,u)||\nabla u|\,\mathrm{d}x\mathrm{d}t\\
			&\leq\iint_{\{s-k\leq|u|\leq s\}}|\Phi(x,t,u)||\nabla u|\,\mathrm{d}x\mathrm{d}t,
		\end{aligned}
	\end{equation}
	and
	\begin{equation}
		I_{3}^{2}\leq\frac{1}{k}\iint_{\{s-k\leq|v|\leq s\}}|\Phi(x,t,v)||\nabla v|\,\mathrm{d}x\mathrm{d}t.
	\end{equation}
	By the local Lipschitz condition of the function $\Phi$, there exists $L_s \in L^{p'}(Q)$ such that
	\begin{align*}
		I_{3}^{3}&=\frac{1}{k}\iint_{\{|v|\leq s\}\bigcap\{|u|\leq s\}}|\Phi(x,t,u)-\Phi(x,t,v)||\nabla T_{k}(u-v)|\,\mathrm{d}x\mathrm{d}t\\
		&\leq\frac{1}{k}\iint_{\{0\leq|T_{s}(v)-T_{s}(u)|<k\}}L_{s}(x,t)|T_{s}(u)-T_{s}(v)||\nabla T_{k}(T_{s}(u)-T_{s}(v))|\,\mathrm{d}x\mathrm{d}t\\
		&\leq\iint_{\{0\leq|T_{s}(v)-T_{s}(u)|<k\}}L_{s}(x,t)|\nabla T_{k}(T_{s}(u)-T_{s}(v))|\,\mathrm{d}x\mathrm{d}t\\
		&\leq\iint_{\{0\leq|T_{s}(v)-T_{s}(u)|<k\}}L_{s}(x,t)(|\nabla T_{s}(u)|+|\nabla T_{s}(v)|)\,\mathrm{d}x\mathrm{d}t,
	\end{align*}
	where $L_s \in L^p(Q)$. By \eqref{beforetau0}, $L_s(x,t) \bigl( |\nabla T_s(u)| + |\nabla T_s(v)| \bigr) \in L^1(Q)$. Moreover, as $k \to 0$,
	\[
	\chi_{\{0 \leq |T_s(u) - T_s(v)| < k\}} \to 0 \quad \text{a.e. in } Q,
	\]
	and this sequence is bounded. Then by the Lebesgue dominated convergence theorem, we obtain
	\[
	\lim_{k \to 0} I_3^3 = 0.
	\]
	Consequently,
	\[
	\limsup_{\sigma \to 0} |I_3| \leq \frac{M_1}{k} \Gamma(u,v,s,k) + \omega(k),
	\]
	where $M_1$ is a constant independent of $s$, $\sigma$, and $k$, and $\omega$ is a positive function satisfying
	\[
	\lim_{k \to 0} \omega(k) = 0.
	\] 
	In \eqref{i012345}, we first take the upper limit as $\sigma \to 0$, and then take the upper limit as $k \to 0$. From \eqref{tsint0}, we obtain
	\begin{equation}
		\begin{aligned}[b]
			&\int_{\Omega}|T_{s}(u)-T_{s}(v)|\,\mathrm{d}x \\
			\leq&-\liminf_{\sigma\rightarrow 0}\limsup_{k\rightarrow 0}|I_{1}|+\limsup_{\sigma\rightarrow0}\limsup_{k\rightarrow0}(|I_{2}|+|I_{4}|)\\
			&+\limsup_{\sigma\rightarrow0}
			\limsup_{k\rightarrow0}|I_{3}|+\limsup_{\sigma\rightarrow0}
			\limsup_{k\rightarrow0}|I_{5}|\\
			\leq& M_{1}\limsup_{k\rightarrow0}\frac{1}{k}\Gamma(u,v,s,k)+M_{1}\frac{1}{\sigma}\Gamma(u,v,s,\sigma)+\omega(s)
		\end{aligned}
	\end{equation}
	For any $s > 0$ and any $t \in (0, T)$, we have $\omega(s) \to 0$ as $s \to 0$.
	Since $u$ (resp. $v$) is finite almost everywhere in $Q$, as $s \to \infty$, $T_s(u)(t)$ (resp. $T_s(v)(t)$) converges almost everywhere to $u(t)$ (resp. $v(t)$) for any $t \in (0, T)$. By Fatou's lemma, taking the lower limit as $s \to \infty$, we obtain
	\begin{equation}
		\int_{\Omega}|u(t)-v(t)|\,\mathrm{d}x\leq 2M_{1}\liminf_{s\rightarrow\infty}\limsup_{k\rightarrow0}\frac{1}{k}\Gamma(u,v,s,k).
	\end{equation}
	By Lemma\ref{lemmaboccardo}, we have
	\begin{equation}
		\int_{\Omega} |u(t) - v(t)| \, \mathrm{d}x = 0.
	\end{equation}
	Hence, for any $t \in (0, T)$, $u = v$ a.e. in $Q$. That is, equation \eqref{equations} admits a unique renormalized solution.
\end{proof}
\section*{Acknowledgments}
	
\newpage
\bibliographystyle{plain}
\bibliography{ref.bib}

@book{Adams1975,
	author = {Adams, R. A.},
	title = {Sobolev Spaces},
	publisher = {Academic Press},
	address = {New York},
	year = {1975}
}

@article{Blanchard1997,
	author = {Blanchard, D. and Murat, F.},
	title = {Renormalized solution for nonlinear parabolic problems with $L^1$ data: existence and uniqueness},
	journal = {Proceedings of the Royal Society of Edinburgh},
	year = {1997},
	volume = {127A},
	pages = {1137-1152}
}

@article{Blanchard2001,
	author = {Blanchard, D. and Murat, F. and Redwane, H.},
	title = {Existence and uniqueness of a renormalized solution for a fairly general class of nonlinear parabolic problems},
	journal = {Journal of Differential Equations},
	year = {2001},
	volume = {177},
	pages = {331-374}
}

@article{Boccardo1989,
	author = {Boccardo, L. and Gallouët, T.},
	title = {On some nonlinear elliptic and parabolic equations involving measure data},
	journal = {Journal of Functional Analysis},
	year = {1989},
	volume = {87},
	pages = {149-169}
}

@article{Boccardo1997,
	author = {Boccardo, L. and Dall'Aglio, A. and Gallouët, T. and others},
	title = {Nonlinear parabolic equations with measure data},
	journal = {Journal of Functional Analysis},
	year = {1997},
	volume = {147},
	number = {1},
	pages = {237-258}
}

@article{Boccardo2003,
	author = {Boccardo, L. and Orsina, L. and Porretta, A.},
	title = {Some noncoercive parabolic equations with lower order terms in divergence form},
	journal = {Journal of Evolution Equations},
	year = {2003},
	volume = {3},
	number = {3},
	pages = {407-418}
}

@article{Boccardo1988,
	author = {Boccardo, L. and Murat, F. and Puel, J. P.},
	title = {Existence of bounded solutions for nonlinear elliptic unilateral problems},
	journal = {Annali di Matematica Pura ed Applicata},
	year = {1988},
	volume = {152},
	pages = {183-196}
}

@article{Bottaro1973,
	author = {Bottaro, G. and Marina, M. E.},
	title = {Problemi di Dirichlet per equazioni ellittiche di tipo variazionale su insiemi non limitati},
	journal = {Bollettino della Unione Matematica Italiana},
	year = {1973},
	volume = {8},
	pages = {46-56}
}

@article{DalMaso1999,
	author = {Dal Maso, G. and Murat, F. and Orsina, L. and others},
	title = {Renormalized solutions of elliptic equations with general measure data},
	journal = {Annali della Scuola Normale Superiore di Pisa - Classe di Scienze},
	year = {1999},
	volume = {28},
	number = {4},
	pages = {741-808}
}

@article{DiNardo2010,
	author = {Di Nardo, R.},
	title = {Nonlinear parabolic equations with a lower order term and $L^1$ data},
	journal = {Communications on Pure and Applied Analysis},
	year = {2010},
	volume = {9},
	number = {4},
	pages = {929-942}
}

@article{DiNardo2011,
	author = {Di Nardo, R. and Feo, F. and Guibé, O.},
	title = {Existence result for nonlinear parabolic equations with lower order terms},
	journal = {Analysis and Applications},
	year = {2011},
	volume = {9},
	number = {2},
	pages = {161-186}
}

@article{DiPerna1989a,
	author = {DiPerna, R. J. and Lions, P. L.},
	title = {Ordinary differential equations, Sobolev spaces and transport theory},
	journal = {Inventiones Mathematicae},
	year = {1989a},
	volume = {98},
	number = {3},
	pages = {511-547}
}

@article{DiPerna1989b,
	author = {DiPerna, R. J. and Lions, P. L.},
	title = {On the Cauchy problem for Boltzmann equations: Global existence and weak stability},
	journal = {Annals of Mathematics},
	year = {1989b},
	volume = {130},
	number = {1},
	pages = {321-366}
}

@article{Hunt1966,
	author = {Hunt, R.},
	title = {On $L(p,q)$ spaces},
	journal = {Enseignement Mathématique},
	year = {1966},
	volume = {12},
	pages = {249-276}
}

@article{Landes1981,
	author = {Landes, R.},
	title = {On the existence of weak solutions for quasilinear parabolic initial boundary value problems},
	journal = {Proceedings of the Royal Society of Edinburgh},
	year = {1981},
	volume = {89A},
	pages = {217-237}
}

@book{Lions1969,
	author = {Lions, J. L.},
	title = {Quelques méthodes de résolutions des problèmes aux limites non linéaires},
	publisher = {Dunod et Gauthier-Villars},
	address = {Paris},
	year = {1969}
}

@article{Lorentz1950,
	author = {Lorentz, G. G.},
	title = {Some new functional spaces},
	journal = {Annals of Mathematics},
	year = {1950},
	volume = {51},
	pages = {37-55}
}

@unpublished{Murat1993,
	author = {Murat, F.},
	title = {Soluciones renormalizadas de EDP elípticas no lineales},
	note = {Laboratoire d'Analyse Numérique, Paris VI, cours à l'Université de Séville},
	year = {1993}
}

@article{Nirenberg1959,
	author = {Nirenberg, L.},
	title = {On elliptic partial differential equations},
	journal = {Annali della Scuola Normale Superiore di Pisa - Classe di Scienze},
	year = {1959},
	volume = {2},
	pages = {115-162}
}

@article{Prignet1995,
	author = {Prignet, A.},
	title = {Remarks on existence and uniqueness of solutions of elliptic problems with right-hand side measures},
	journal = {Rendiconti Lincei - Matematica e Applicazioni},
	year = {1995},
	volume = {15},
	number = {3},
	pages = {321-337}
}

@article{Prignet1997,
	author = {Prignet, A.},
	title = {Existence and uniqueness of ``entropy'' solutions of parabolic problems with $L^1$ data},
	journal = {Nonlinear Analysis},
	year = {1997},
	volume = {28},
	number = {12},
	pages = {1943-1954}
}

@article{Porretta1999,
	author = {Porretta, A.},
	title = {Existence results of nonlinear parabolic equations via strong convergence of truncations},
	journal = {Annali di Matematica Pura ed Applicata},
	year = {1999},
	volume = {177},
	number = {1},
	pages = {143-172}
}

@article{Porzio1999,
	author = {Porzio, M. M.},
	title = {Existence of solutions for some ``noncoercive'' parabolic equations},
	journal = {Discrete and Continuous Dynamical Systems},
	year = {1999},
	volume = {5},
	number = {3},
	pages = {553-568}
}

@article{Serrin1964,
	author = {Serrin, J.},
	title = {Pathological solution of elliptic differential equations},
	journal = {Annali della Scuola Normale Superiore di Pisa - Classe di Scienze},
	year = {1964},
	volume = {18},
	number = {3},
	pages = {385-387}
}

@article{Simon1987,
	author = {Simon, J.},
	title = {Compact sets in the space $L^p(0,T;B)$},
	journal = {Annali di Matematica Pura ed Applicata},
	year = {1987},
	volume = {146},
	pages = {65-96}
}

@book{Ziemer1989,
	author = {Ziemer, W. P.},
	title = {Weakly Differentiable Functions},
	series = {Graduate Texts in Mathematics},
	volume = {120},
	publisher = {Springer-Verlag},
	address = {New York},
	year = {1989}
}
\end{document}